\documentclass[a4paper,10pt]{amsart}

\usepackage{amsmath}
\usepackage{amsthm}
\usepackage{amssymb}
\usepackage{amsfonts}
\usepackage{bbm}
\usepackage{mathrsfs}
\usepackage{enumitem} % more enumerate/itemize options
\usepackage{xcolor}
\usepackage[bookmarks=true,hyperindex,pdftex,colorlinks,citecolor=blue,linkcolor=cyan]{hyperref}
\usepackage{marginnote} % add margin notes
\usepackage{fancyhdr}
\usepackage{soul}

\DeclareMathOperator{\dist}{dist}                           % distance between sets
\DeclareMathOperator{\lspan}{span}                          % linear span
\DeclareMathOperator{\conv}{conv}                           % convex hull
\DeclareMathOperator{\supp}{supp}                           % support
\DeclareMathOperator{\rad}{rad}                             % radius
\DeclareMathOperator{\Lip}{Lip}                             % Lipschitz functions
\DeclareMathOperator{\aspan}{span}

\newcommand{\N}{\mathbb{N}}   
\newcommand{\NN}{\mathbb{N}}            % natural numbers
\newcommand{\ZZ}{\mathbb{Z}}             % integer numbers
\newcommand{\RR}{\mathbb{R}}             % real numbers
\newcommand{\abs}[1]{\left|{#1}\right|}                     % absolute value
\newcommand{\pare}[1]{\left({#1}\right)}                    % parentheses
\newcommand{\set}[1]{\left\{{#1}\right\}}                   % set by extension
\newcommand{\norm}[1]{\left\|{#1}\right\|}                  % norm
\newcommand{\duality}[1]{\left<{#1}\right>}                 % dual action
\newcommand{\cl}[1]{\overline{#1}}                          % closure
\newcommand{\restrict}{\mathord{\upharpoonright}}           % restriction symbol

\newcommand{\lipfree}[1]{\mathcal{F}({#1})}                 % Lipschitz-free space
\newcommand{\meas}[1]{\mathcal{M}({#1})}                    % Radon measure space
\newcommand{\free}[1]{\ensuremath{\mathcal{F}({#1})}}

\newcommand{\ip}[2]{\ensuremath{\left\langle{#1},{#2}\right\rangle}}% inner product
\newcommand{\n}[1]{\norm{#1}}
\newcommand{\enorm}[2]{\norm{#1}_{\ell_2^{#2}}}

\renewcommand{\leq}{\leqslant}
\renewcommand{\geq}{\geqslant}

\theoremstyle{plain}
\newtheorem{theorem}{Theorem}[section]
\newtheorem{lemma}[theorem]{Lemma}
\newtheorem{corollary}[theorem]{Corollary}
\newtheorem{proposition}[theorem]{Proposition}

\newtheorem*{claim*}{Claim}
\newtheorem{fact}[theorem]{Fact}
\newtheorem{question}{Question}

\newtheorem{maintheorem}{Theorem}

\newtheorem{innercustomthm}{Theorem}
\newenvironment{customtheorem}[1]
  {\renewcommand\theinnercustomthm{#1}\innercustomthm}
  {\endinnercustomthm}

\theoremstyle{definition}
\newtheorem*{definition*}{Definition}
\newtheorem{definition}[theorem]{Definition}

\theoremstyle{remark}
\newtheorem{remark}[theorem]{Remark}

\begin{document}
\title{Pe\l czy\'nski's property (V$^*$) in Lipschitz-free spaces}

\author[R. J. Aliaga]{Ram\'on J. Aliaga}
\address[R. J. Aliaga]{Instituto Universitario de Matem\'atica Pura y Aplicada,
Universitat Polit\`ecnica de Val\`encia,
Camino de Vera S/N,
46022 Valencia, Spain}
\email{raalva@upv.es}

\author[E. Perneck\'a]{Eva Perneck\'a}
\address[E. Perneck\'a]{Faculty of Information Technology, Czech Technical University in Prague, Th\'akurova 9, 160 00, Prague 6, Czech Republic}
\email{perneeva@fit.cvut.cz}

\author[A. Quero]{Alicia Quero}
\address[A. Quero]{\textit{Present address:} Facultad de Ciencias Matem\'aticas, Universitat de Val\`encia, Doctor Moliner 50, 46100 Burjasot (Valencia), Spain
\newline
Faculty of Information Technology, Czech Technical University in Prague, Th\'akurova 9, 160 00, Prague 6, Czech Republic
}
\email{alicia.quero@uv.es}

\begin{abstract}
We prove that Pe\l czy\'nski's property (V$^*$) is locally determined for Lipschitz-free spaces, and obtain several sufficient conditions for it to hold. We deduce that $\mathcal{F}(M)$ has property (V$^*$) when the complete metric space $M$ is locally compact and purely 1-unrectifiable, a Hilbert space, or belongs to a class of Carnot-Carath\'eodory spaces satisfying a bi-H\"older condition, including Carnot groups.
\end{abstract}

%46B03 (Isomorphic theory (including renorming) of Banach spaces)
%46B50 (Compactness in Banach (or normed) spaces)
%53C17 (Sub-Riemannian geometry)


\maketitle

\section{Introduction}

This paper deals with the isomorphic theory of Lipschitz-free Banach spaces. Given a metric space $M$, its Lipschitz-free space $\lipfree{M}$ is a canonical linearization of $M$: a Banach space built around $M$ in such a way that $M$ is a linearly independent and linearly dense subset of $\lipfree{M}$, and any Lipschitz map $f:M\to N$ between metric spaces can be extended into an operator $\widehat{f}:\lipfree{M}\to\lipfree{N}$. See Section \ref{sec:lipfree} for a more detailed definition. The class of Lipschitz-free spaces is currently among the most studied within Banach space theory.

Properties (V) and (V$^*$), dual to each other, were introduced by Pe\l czy\'nski in \cite{Pel62}. For a Banach space $X$, property (V$^*$) relates weak compactness in $X$ with the convergence of weakly unconditionally Cauchy series in $X^*$; for property (V), the roles of $X$ and $X^*$ are reversed. Again, precise definitions will follow. On the other hand, $X$ is said to be \emph{weakly sequentially complete} if every weakly Cauchy sequence in $X$ is weakly convergent; recall that a sequence $(x_n)_n$ is \emph{weakly Cauchy} if the sequence $(\duality{x^*,x_n})_n$ is Cauchy (i.e. convergent) for every $x^*\in X^*$.

Our main motivation is the study of the following question.

\begin{question}
\label{q:equiv v* wsc c0}
Are any of the following properties equivalent for Lipschitz-free spaces?
\begin{enumerate}[label={\upshape{(\roman*)}}]
\item property (V$^*$),
\item weak sequential completeness,
\item not containing (isomorphic copies of) $c_0$.
\end{enumerate}
\end{question}

The implications (i)$\Rightarrow$(ii)$\Rightarrow$(iii) hold for all Banach spaces, and neither of them can be reversed in general. Indeed, James space witnesses (iii)$\not\Rightarrow$(ii), while Bourgain and Delbaen constructed in \cite{BourgainDelbaen} a space satisfying (ii) but not (i). However, properties (i)-(iii) are known to be equivalent for certain classes of Banach spaces with additional structure, such as Banach lattices \cite{SaabSaab} or, more generally, complemented subspaces of Banach lattices \cite{Bombal}. No counterexample to either reverse implication is currently known within the class of Lipschitz-free spaces. On top of that, Lipschitz-free spaces already have a history of providing unexpected equivalences between prominent Banach space properties. For instance, the equivalence of the Schur property, the Radon-Nikod\'ym property, and not containing $L_1$ is established in \cite{AGPP}. See also \cite{AGP}, where it is shown that many classical Banach space properties are equivalent to separability. In light of all of this, Question \ref{q:equiv v* wsc c0} is a very natural one.

Here, we focus on property (V$^*$) and the possible equivalence (i)$\Leftrightarrow$(ii). We will now describe the partial results obtained in this paper.
It is shown in \cite{ANPP} that certain Banach space properties are \emph{compactly determined} for Lipschitz-free spaces, meaning that they hold for $\lipfree{M}$ as soon as they hold for $\lipfree{K}$ for each compact $K\subset M$. Compactly determined properties include weak sequential completeness \cite[Corollary 2.4]{ANPP} and not containing $c_0$ \cite[Theorem 2.11]{ANPP}. Thus, if property (V$^*$) really is equivalent to weak sequential completeness, then the following question should have a positive answer.

\begin{question}
Is property (V$^*$) compactly determined for Lipschitz-free spaces?
\end{question}

We do not know the full answer to this question, but we provide several partial answers. First, we prove that property (V$^*$) is subject to a weaker determinacy principle that we call \emph{local determination}. Namely, $\lipfree{M}$ has property (V$^*$) if every point $x\in M$ has a neighborhood $U$ such that $\lipfree{U}$ has property (V$^*$).

\begin{maintheorem}
\label{mth:locally determined}
Property (V$^*$) is locally determined for Lipschitz-free spaces.
\end{maintheorem}

A Banach space has the \emph{Schur property} if every weakly convergent sequence is norm convergent. Clearly, every such space is weakly sequentially complete. In \cite{AGPP}, Lipschitz-free spaces $\lipfree{M}$ with the Schur property are characterized: this holds precisely when the completion of $M$ is \emph{purely 1-unrectifiable}, meaning that it contains no bi-Lipschitz copy of a subset of $\RR$ with positive measure. Thus, if Question \ref{q:equiv v* wsc c0} has a positive answer, then this should also be true:

\begin{question}
If $M$ is a complete, purely 1-unrectifiable metric space, does $\lipfree{M}$ have property (V$^*$)?
\end{question}

We cannot answer this in general. However, compact and local determination are equivalent for $\lipfree{M}$ when $M$ is locally compact, so property (V$^*$) is compactly determined in that particular case. This is enough to establish the following partial result.

\begin{maintheorem}
\label{mth:locally compact p1u}
If $M$ is a complete, locally compact, purely 1-unrectifiable metric space then $\lipfree{M}$ has property (V$^*$).
\end{maintheorem}

Both Theorem \ref{mth:locally determined} and Theorem \ref{mth:locally compact p1u} imply immediately, for instance, that $\lipfree{M}$ has property (V$^*$) when $M$ is discrete and complete.

Yet another situation concerns superreflexive Banach spaces $X$ (as metric spaces, endowed with the norm metric). The main result of \cite{KP} shows that $\lipfree{M}$ has property (V$^*$) whenever $M$ is a compact subset of $X$. This is used in \cite{ANPP} to prove that $\lipfree{X}$ is weakly sequentially complete by compact determination. Hence the obvious question:

\begin{question}
If $X$ is a superreflexive Banach space, does $\lipfree{X}$ have property (V$^*$)?
\end{question}

Again, we are not able to prove this in general. But we are able to find sufficient technical conditions for the compact determination of property (V$^*$) (see Proposition \ref{pr:sufficient conditions}), and we show that they hold for the class of superreflexive spaces with the best possible geometric behavior:

\begin{maintheorem}
\label{mth:hilbert v*}
If $X$ is a Hilbert space, then $\lipfree{X}$ has property (V$^*$).
\end{maintheorem}

The class of superreflexive spaces strongly contrasts with the discrete metric spaces mentioned above. And in fact, while property (V$^*$) of the Lipschitz-free space over a discrete metric space follows from reduction to finite metric spaces, the result in \cite{KP} is based on the abundance of curves and rich differentiation structure in superreflexive spaces. These allow estimates of incremental quotients of Lipschitz functions by integrals of their derivatives. It might be expected that arguments analogous to those in \cite{KP} could be applied in other metric settings possessing a similar structure. Indeed, such an agreeable framework is provided by certain Carnot-Carath\'eodory spaces, including the Carnot groups.

\begin{maintheorem}
\label{mth:carnot v*}
If $G$ is a Carnot group, then $\lipfree{G}$ has property (V$^*$).
\end{maintheorem}

See Theorem \ref{thm:carnot_compact} for the more general result.

Finally, we also consider briefly Silber's property (R), introduced recently in \cite{Silber} as somewhat of an analog to property (V$^*$) for super weak compactness, and we show that all Lipschitz-free spaces covered by Theorem \ref{thm:carnot_compact} satisfy property (R) as well. The same holds for Lipschitz-free spaces in Theorems \ref{mth:locally compact p1u} and \ref{mth:hilbert v*}, but this already follows from the results in \cite{Silber}.

\medskip
\noindent \textbf{Notation.} We use standard notation throughout the paper. $X$ usually stands for a Banach space, and $B_X$ and $S_X$ denote its closed unit ball and unit sphere. By ``operator'' we mean a bounded linear operator between Banach spaces. On the other hand, $M$ stands for a complete metric space with metric $d$, $B(x,r)$ is the closed ball with center $x\in M$ and radius $r\geq 0$, and we use the notation
$$
[E]_r = \set{x\in M \,:\, d(x,E)\leq r}
$$
for $E\subset M$ and $r\geq 0$. We assume that $M$ is \emph{pointed}, that is, it comes equipped with a distinguished ``base point'' $0\in M$, and write
$$
\rad(E) = \sup\set{d(x,0) \,:\, x\in E}
$$
for the \emph{radius} of $E\subset M$.

\subsection{Basics on Lipschitz-free spaces}
\label{sec:lipfree}

We consider the Banach space $\Lip_0(M)$ of all Lipschitz functions $f:M\to\RR$ such that $f(0)=0$, endowed with the Lipschitz norm
$$
\Lip(f) = \sup\set{\frac{f(x)-f(y)}{d(x,y)} \,:\, x\neq y\in M} .
$$
The evaluation functionals $\delta(x):f\mapsto f(x)$ belong to $\Lip_0(M)^*$ for each $x\in M$, so this defines a mapping $\delta:M\to\Lip_0(M)^*$ which is, in fact, an isometric embedding. The \emph{Lipschitz-free space} over $M$ is defined as
$$
\lipfree{M} = \cl{\lspan}\,\delta(M) \subset \Lip_0(M)^* .
$$
Then $\lipfree{M}$ is an isometric predual of $\Lip_0(M)$, such that the weak$^*$ topology on $B_{\Lip_0(M)}$ coincides with the topology of pointwise convergence. Changes of base point result in linearly isometric Lipschitz and Lipschitz-free spaces. The standard reference for these spaces is \cite{Weaver2}.

The space $\lipfree{M}$ can be thought of as a linearization of $M$ (being the closed linear hull of its isometric copy $\delta(M)$) with the following extension property: any Lipschitz mapping $f:M\to X$ into a Banach space $X$ with $f(0)=0$ can be linearized into an operator $F:\lipfree{M}\to X$ such that $F(\delta(x))=f(x)$ for $x\in M$. In particular, any Lipschitz mapping $f:M\to N$ between pointed metric spaces with $f(0_M)=0_N$ can be linearized into an operator $\widehat{f}:\lipfree{M}\to\lipfree{N}$ such that $\widehat{f}(\delta(x))=\delta(f(x))$ for $x\in M$. 

If $N\subset M$ is closed and contains $0$, then $\lipfree{N}$ can be identified isometrically with the subspace $\cl{\lspan}\,\delta(N)$ of $\lipfree{M}$. The \emph{support} of $m\in\lipfree{M}$, denoted by $\supp(m)$, is defined as the smallest closed set $S\subset M$ such that $m\in\lipfree{S\cup\set{0}}$; its existence is proved in \cite{APPP}. It has the property that $\duality{f,m}=\duality{g,m}$ whenever $f,g\in\Lip_0(M)$ agree on $\supp(m)$ \cite[Proposition 2.6]{APPP}. The set of elements of $\lipfree{M}$ whose support is finite is precisely $\lspan\,\delta(M)$.

Elements of $\lipfree{M}$ of the form
$$
m_{xy} = \frac{\delta(x)-\delta(y)}{d(x,y)} \in S_{\lipfree{M}}
$$
for $x\neq y\in M$ are called \emph{(elementary) molecules}. Since the molecules norm $\Lip_0(M)$, they span $B_{\lipfree{M}}$. As a consequence, given $\varepsilon>0$, any $m\in\lipfree{M}$ can be represented as a series of molecules
\begin{equation}\label{eq:sum of molecules}
m = \sum_{n=1}^\infty a_nm_{x_ny_n}
\end{equation}
such that $a_n\geq 0$ and $\sum_na_n\leq\norm{m}+\varepsilon$. If $m$ has finite support then the sum in \eqref{eq:sum of molecules} can be chosen to be finite and such that $\sum_na_n=\norm{m}$ (we call such a representation \emph{optimal}).

Last, we introduce a special class of operators on Lipschitz and Lipschitz-free spaces. Suppose that $h:M\to\RR$ is Lipschitz and has bounded support, and let $N$ be a closed subset of $M$ containing $\supp(h)$. Then the mapping $W_h:\Lip_0(N)\to\Lip_0(M)$ given by
\begin{equation}\label{eq:weighting operator}
W_h(f)(x) = \begin{cases}
f(x)h(x) &\text{if $x\in N$} \\
0 &\text{otherwise}
\end{cases}
\end{equation}
is a weak$^*$-weak$^*$-continuous operator whose norm satisfies
$$
\norm{W_h}\leq \norm{h}_\infty+\rad(\supp(h))\Lip(h)
$$
by \cite[Lemma 2.3]{APPP}. Its adjoint can thus be restricted to $\lipfree{M}$ to yield an operator $W_h^*:\lipfree{M}\to\lipfree{N}$ with $\norm{W_h}=\norm{W_h^*}$; if we see $\lipfree{N}$ as a subspace of $\lipfree{M}$, then $\duality{f,W_h^*(m)}=\duality{fh,m}$ for $m\in\lipfree{M}$ and $f\in\Lip_0(N)$. The operators $W_h$ and $W_h^*$ are referred to as \emph{weighting operators}.

\subsection{Basics on properties (V*) and (R)}

\begin{definition}
\label{def:WUC}
Let $(x_n)$ be a sequence in a Banach space $X$. We say that the (formal) series $\sum_nx_n$ is \emph{weakly unconditionally convergent} (or \emph{weakly unconditionally Cauchy}), \emph{WUC} for short, if $\sum_n\abs{\duality{x^*,x_n}}<\infty$ for every $x^*\in X^*$.
\end{definition}

The following easy facts about WUC series follow from either the Banach-Steinhaus theorem or e.g.\ \cite[Proposition II.D.4]{Wojt}.

\begin{fact}
\label{fact:wuc}
Let $X$ be a Banach space.
\begin{enumerate}[label={\upshape{(\alph*)}}]
\item\label{wuc operator} Operators map WUC series to WUC series.
\item\label{wuc bound} If $\sum_nx_n$ is WUC in $X$ then there exists $C<\infty$ such that $\sum_n\abs{\duality{x^*,x_n}}\leq C\norm{x^*}$ for all $x^*\in X^*$. In particular, $\norm{x_n}\leq C$ for all $n$.
\item\label{w*uc} A formal series $\sum_n x^*_n$ in $X^*$ is WUC whenever $\sum_n\abs{\duality{x^*_n,x}}<\infty$ for all $x\in X$ (that is, in the definition it is enough to consider functionals in the predual).
\item\label{wuc c0} $\sum_n x_n$ is a WUC series in X if and only if there exists an operator ${S:c_0\to X}$ such that $Se_n=x_n$, where $e_n$ is the canonical $c_0$ basis.
\item\label{wuc l1} $\sum_n x^*_n$ is a WUC series in $X^*$ if and only if there exists an operator $T:X\to\ell_1$ such that $x^*_n=e_n\circ T$, where $e_n$ is the canonical $c_0$ basis.
\end{enumerate}
\end{fact}

Subsequent discussions will involve the following four special classes of sets in Banach spaces.

\begin{definition}
\label{def:sets}
Let $\Gamma$ be a subset of a Banach space $X$. We say that $\Gamma$ is...
\begin{itemize}
\item ...a \emph{(V$^*$)-set (with respect to $X$)} if
$$
\lim_{n\to\infty}\sup_{x\in\Gamma}\abs{\duality{x,x_n^*}} = 0
$$
for every WUC series $\sum_nx_n^\ast$ in $X^\ast$.
\item ...\emph{weakly precompact} if every sequence in $\Gamma$ has a weakly Cauchy subsequence.
\item ...\emph{relatively weakly compact} if its weak closure is weakly compact.
\item ...\emph{relatively super weakly compact (RSWC)} if $\Gamma^\mathcal{U}$ is relatively weakly compact in $X^\mathcal{U}$ for every free ultrafilter $\mathcal{U}$ on $\NN$.
\end{itemize}
\end{definition}

Here, the ultrapower $X^\mathcal{U}$ is the quotient of $\ell_\infty(X)$ by the subspace of those $(x_n)_n$ such that $\lim_{\mathcal{U},n}x_n=0$, and $\Gamma^\mathcal{U}$ is the subset of elements of $X^\mathcal{U}$ having some representative $(x_n)_n$ such that $x_n\in\Gamma$ for all $n$. For further reference, see e.g.\ \cite{LancienRaja}.

Some of these classes of sets have useful alternative characterizations. For instance, by Rosenthal's $\ell_1$ theorem, a set is weakly precompact precisely when it is bounded and contains no equivalent $\ell_1$ basis. (V$^*$)-sets have several useful characterizations as well. We will use these interchangeably without mention.

\begin{proposition}[{\cite[Proposition 1.1]{Bombal}}]
\label{pr:v* sets}
For a bounded subset $\Gamma$ of a Banach space $X$, the following are equivalent:
\begin{enumerate}[label={\upshape{(\roman*)}}]
\item $\Gamma$ is a (V$^*$)-set with respect to $X$,
\item\label{pr:v* sets:l1 basis} $\Gamma$ does not contain an equivalent $\ell_1$ basis whose closed linear span is complemented in $X$,
\item\label{pr:v* sets:l1 op} $T(\Gamma)$ is relatively compact for every operator $T:X\to\ell_1$.
\end{enumerate}
\end{proposition}

We will address RSWC sets exclusively through the following characterization due to Lancien and Raja, which is only valid in the convex case.

\begin{proposition}[{\cite[Theorem 3.8]{LancienRaja}}]
\label{pr:lancien_raja}
Let $\Gamma$ be a closed convex subset of a Banach space $X$. Then $\Gamma$ is not RSWC if and only if there exists $\xi>0$ such that for every $n\in\NN$ there exist $n$-tuples $(x_1,\ldots,x_n)$ in $\Gamma$ and $(f_1,\ldots,f_n)$ in $B_{X^*}$ such that
$$
f_k(x_j) = \begin{cases}
0 &\text{, if $k>j$} \\
\xi &\text{, if $k\leq j$}
\end{cases} .
$$
\end{proposition}

All four classes of sets defined in Definition \ref{def:sets} share many common properties: they are stable under the most common Banach space operations, they are all countably determined, and they all satisfy some form of Grothendieck's lemma. This is summarized as follows.

\begin{fact}
\label{fact:sets}
Let $\mathcal{C}_X$ be the class of all subsets of a Banach space $X$ of one of the four types in Definition \ref{def:sets}. Then:
\begin{enumerate}[label={\upshape{(\alph*)}}]
\item\label{fact:sets:a} All $\Gamma\in\mathcal{C}_X$ are bounded.
\item If $\Gamma\in\mathcal{C}_X$ and $\Gamma'\subset\Gamma$ then $\Gamma'\in\mathcal{C}_X$.
\item If $\Gamma,\Gamma'\in\mathcal{C}_X$ then $\Gamma+\Gamma'\in\mathcal{C}_X$.
\item\label{fact:sets:d} If $\Gamma\in\mathcal{C}_X$ and $T:X\to Y$ is an operator then $T(\Gamma)\in\mathcal{C}_Y$.
\item\label{fact:sets:conv} If $\Gamma\in\mathcal{C}_X$ then $\conv\,\Gamma\in\mathcal{C}_X$.
\item\label{fact:sets:countable} $\Gamma\in\mathcal{C}_X$ if and only if $A\in\mathcal{C}_X$ for every countable subset $A\subset\Gamma$.
\item\label{fact:sets:grothendieck} $\Gamma\in\mathcal{C}_X$ if and only if for every $\varepsilon>0$ there exists $\Gamma_\varepsilon\in\mathcal{C}_X$ such that $\Gamma\subset\Gamma_\varepsilon+\varepsilon B_X$.
\end{enumerate}
\end{fact}

\noindent Indeed, statements \ref{fact:sets:a}-\ref{fact:sets:d} are straightforward to check from the definitions. Property \ref{fact:sets:conv} is also straightforward for (V$^*$)-sets and relatively weakly compact sets; for weakly precompact sets, an argument is provided by Rosenthal in \cite[Addendum]{Rosenthal}; and for RSWC sets, this was proved by Tu in \cite[Theorem 3.1]{Tu}. Property \ref{fact:sets:countable} is either obvious or follows from the Eberlein-\v{S}mulyan theorem. Statement \ref{fact:sets:grothendieck} for relatively weakly compact sets is the classical Grothendieck lemma (see e.g.\ \cite[Lemma II.C.7]{Wojt}); it was proved for (V$^*$)-sets in \cite[Corollary 1.7]{Bombal} and for RSWC sets in \cite[Theorem 4.1]{CCLTZ}; and is easy to check for weakly precompact sets using their characterization in terms of $\ell_1$ bases.

Finally, we have the following dependency relations between the four notions:
\begin{equation}
\label{eq:set implications}
\text{RSWC} \Longrightarrow \text{relatively weakly compact} \Longrightarrow \text{weakly precompact} \Longrightarrow \text{(V$^*$)-set}
\end{equation}
Indeed, the first one follows from the definition, the second one from the Eberlein-\v{S}mulyan theorem, and the last one from Proposition \ref{pr:v* sets}\ref{pr:v* sets:l1 op}. None of them can be reversed in general: for the first one, this is witnessed e.g.\ by the unit ball of reflexive, non-superreflexive Banach spaces; for the second one, consider a weakly Cauchy sequence that does not converge weakly; the last one is witnessed by any uncomplemented $\ell_1$ basis, by Proposition \ref{pr:v* sets}\ref{pr:v* sets:l1 basis}.

Certain Banach spaces do allow some of these implications to be reversed. This motivates the following definitions:

\begin{definition}
We say that a Banach space $X$ has...
\begin{itemize}
\item ...\emph{property (V$^*$)} if every (V$^*$)-set (with respect to $X$) is relatively weakly compact.
\item ...\emph{property (R)} if every weakly precompact set is RSWC; equivalently by \cite{Silber2}, if every relatively weakly compact set is RSWC.
\end{itemize}
\end{definition}

Property (V$^*$) is due to Pe\l czy\'nski \cite{Pel62}, while property (R) was recently introduced by Silber \cite{Silber}.
By \eqref{eq:set implications}, a Banach space has both properties (V$^*$) and (R) if and only if every (V$^*$)-set is RSWC. It follows from Fact \ref{fact:sets} that both properties are preserved by isomorphisms, by passing to closed subspaces, and by finite direct sums. It is easily established that, more generally, the $\ell_1$ sum of arbitrarily many spaces with property (V$^*$) has property (V$^*$) as well \cite[Corollary 2.6]{Emmanuele}. In particular, $\ell_1(\kappa)$ has property (V$^*$) for any cardinal $\kappa$.

These facts are already enough to show that $\lipfree{M}$ has property (V$^*$) for any uniformly discrete metric space $M$ (that is, such that $\inf_{x\neq y\in M}d(x,y)>0$). Indeed, by \cite[Proposition 4.3]{Kalton04}, $\lipfree{M}$ is isomorphic to a subspace of the $\ell_1$ sum of the spaces $\lipfree{B(0,2^n)}$, $n\in\ZZ$, each of which is isomorphic to some $\ell_1(\kappa)$ when $M$ is uniformly discrete. In particular, $\lipfree{M}$ has property (V$^*$) when $M$ is the integer grid in $c_0$, or an infinite dyadic tree. Theorem \ref{mth:locally compact p1u} proves a more general case.

\section{Compact and local determination}

The following concept was introduced implicitly in \cite{ANPP}.

\begin{definition}\label{def:compact determination}
We say that a Banach space property $\mathcal{P}$ is \emph{compactly determined} for a Lipschitz-free space $\lipfree{M}$ if $\lipfree{M}$ satisfies $\mathcal{P}$ whenever $\lipfree{K\cup\set{0}}$ satisfies $\mathcal{P}$ for all compact subsets $K\subset M$. If the space $\lipfree{M}$ is not specified, then this is assumed to hold for all Lipschitz-free spaces.
\end{definition}

Compactly determined properties include weak sequential completeness, the Schur property, the approximation property, the Dunford-Pettis property \cite{ANPP}, and the Radon-Nikod\'ym property \cite{AGPP}. We now introduce a weaker concept.

\begin{definition}\label{def:local determination}
We say that a Banach space property $\mathcal{P}$ is \emph{locally determined} for a Lipschitz-free space $\lipfree{M}$ if $\lipfree{M}$ satisfies $\mathcal{P}$ whenever every $x\in M$ has a neighborhood $U$ such that $\lipfree{\cl{U}\cup\set{0}}$ satisfies $\mathcal{P}$. Again, if the space $\lipfree{M}$ is not specified, then this is assumed to hold for all Lipschitz-free spaces.
\end{definition}

We will only be interested in ``nice'' properties, such as property (V$^*$), that behave well under standard Banach space operations. For the sake of clarity, let us fix a name for them.

\begin{definition}
We say that a Banach space property $\mathcal{P}$ is \emph{standard} if it is stable under isomorphism, finite direct sums and passing to closed subspaces.
\end{definition}

Note that any (non-null) standard property is automatically satisfied by all finite-dimensional Banach spaces.
Besides properties (V$^*$) and (R), weak sequential completeness and not containing $c_0$ are also standard properties.

Standard properties are particularly pleasant to deal with for Lipschitz-free spaces because they are stable under removing a single point from the metric space, making it unnecessary to specify whether the base point is included or not. Indeed, if $M$ is finite then $\lipfree{M}$ and $\lipfree{M\setminus\set{p}}$ are both finite-dimensional for any $p\in M$; on the other hand, if $M$ is infinite then $\lipfree{M}$ and $\lipfree{M\setminus\set{p}}$ are isomorphic by \cite[Lemma 2.8]{AACD3}. Consequently, for such properties we may simply write $\lipfree{K}$ and $\lipfree{\cl{U}}$ in Definitions \ref{def:compact determination} and \ref{def:local determination}. In some of the forthcoming arguments, we will state ``$\lipfree{K}$ has property $\mathcal{P}$'' and ``$\lipfree{K\cup\set{0}}$ has property $\mathcal{P}$'' interchangeably, and only include $0$ explicitly when it becomes essential to the argument.

\begin{lemma}
\label{lm:local to compact}
Let $\mathcal{P}$ be a standard Banach space property, and let $M$ be a metric space. Suppose that every $x\in M$ has a neighborhood $U$ such that $\lipfree{\cl{U}}$ satisfies $\mathcal{P}$. Then every compact $K\subset M$ has a neighborhood $U$ such that $\lipfree{\cl{U}}$ satisfies $\mathcal{P}$, and $\lipfree{K}$ satisfies $\mathcal{P}$ as well.
\end{lemma}

\begin{proof}
Let $K\subset M$ be compact. By assumption, we may find a finite open cover $U_1,\ldots, U_n$ of $K$ such that each $\lipfree{\cl{U}_i}$ has property $\mathcal{P}$ for $i=1,\ldots,n$. Since property $\mathcal{P}$ passes to closed subspaces, we may assume that all $U_i$ are bounded. Clearly, $d(x,M\setminus U_i)>0$ for every $x\in U_i$, thus
$$
\inf_{x\in K} \sum_{i=1}^n d(x,M\setminus U_i) > 0
$$
by the compactness of $K$. Hence, we may find a neighborhood $U$ of $K$ such that $\cl{U}\subset U_1\cup\ldots\cup U_n$ and moreover
$$
\inf_{x\in\cl{U}} \sum_{i=1}^n d(x,M\setminus U_i) > 0 .
$$
By \cite[Lemmas 2.4 and 2.5]{AALMPPV1}, $\lipfree{\cl{U}\cup\set{0}}$ is isomorphic to a complemented subspace of
$$
\bigoplus_{i=1}^n \lipfree{(\cl{U}_i\cap \cl{U}) \cup\set{0}} \subset \bigoplus_{i=1}^n \lipfree{\cl{U}_i \cup\set{0}} .
$$
Because property $\mathcal{P}$ is standard, it follows that $\lipfree{\cl{U}}$, and hence also $\lipfree{K}$, satisfy the property.
\end{proof}

The following is obvious in light of Lemma \ref{lm:local to compact}.

\begin{proposition}
\label{pr:local vs compact determination}
Let $\mathcal{P}$ be a standard Banach space property. If $\mathcal{P}$ is compactly determined, then it is locally determined. For Lipschitz-free spaces over locally compact metric spaces, the converse also holds.
\end{proposition}

\subsection{Property (V*) is locally determined}

This section is devoted to the proof of Theorem \ref{mth:locally determined}. We need to recall another concept, also from \cite{ANPP}.

\begin{definition}
A subset $\Gamma$ of a Lipschitz-free space $\lipfree{M}$ is \emph{tight} if for every $\varepsilon>0$ there exists a compact set $K\subset M$ such that $\Gamma\subset\lipfree{K}+\varepsilon B_{\lipfree{M}}$ (we assume tacitly that $0\in K$).
\end{definition}

That is, tight sets can be approximated uniformly from Lipschitz-free spaces over compact subsets of $M$. In fact, more is true: by \cite[Theorem 3.2]{ANPP}, the approximation can be linear, even a uniform limit of operators.
More precisely, for every $\varepsilon>0$ there is a compact set $K\subset M$ and a sequence $(T_n)$ of operators on $\lipfree{M}$ such that $(T_n(\gamma))$ converges, uniformly on $\gamma\in\Gamma$, to an element $T(\gamma)\in\lipfree{K}$ such that $\norm{T(\gamma)-\gamma}\leq\varepsilon$; clearly $T:\Gamma\to\lipfree{K}$ is a linear mapping.

Since weakly precompact sets are (V$^*$)-sets, the following statement generalizes \cite[Theorem 2.3]{ANPP}.

\begin{theorem}
\label{th:v*_tight}
(V$^*$)-sets in Lipschitz-free spaces are tight.
\end{theorem}

\begin{proof}
Let $\Gamma$ be a (V$^*$)-set in a Lipschitz-free space $\lipfree{M}$. The proof that $\Gamma$ is tight will be an adaptation of \cite[Propositions 3.3 to 3.6]{ANPP}, which is itself an adaptation of Kalton's original \cite[Lemma 4.5]{Kalton04}. We start by proving the following Claim, which will help us reduce the theorem to the case where $M$ is bounded.

\begin{claim*}
For every $\varepsilon>0$ there is $R>0$ such that $\Gamma\subset\lipfree{B(0,R)}+\varepsilon B_{\lipfree{M}}$.
\end{claim*}

Assume otherwise that there is $\varepsilon>0$ such that for every $R>0$ we can find $\gamma\in\Gamma$ with $\dist(\gamma,\lipfree{B(0,R)})>\varepsilon$. Pick $R_0=1$ and, iteratively, once $R_{n-1}$ has been chosen, pick
\begin{itemize}
\item an element $\gamma_n\in\Gamma$ such that $\dist(\gamma_n,\lipfree{B(0,2R_{n-1})})>\varepsilon$,
\item an element $m_n\in\lipfree{M}$ with bounded (for instance, finite) support such that $\norm{m_n-\gamma_n}\leq 2^{-n}\varepsilon$,
\item and $R_n>\rad(\supp(m_n))$.
\end{itemize}
Note that $\dist(m_n,\lipfree{B(0,2R_{n-1})})>\varepsilon-2^{-n}\varepsilon\geq\varepsilon/2$ so, by the Hahn-Banach theorem, there exists $h_n\in B_{\Lip_0(M)}$ such that $h_n=0$ on $B(0,2R_{n-1})$ and $\duality{h_n,m_n}>\varepsilon/2$; in particular, $R_n>2R_{n-1}$. By replacing it with $h_n^+$ or $h_n^-$, we assume that $h_n\geq 0$ and $\abs{\duality{h_n,m_n}}>\varepsilon/4$ instead. Now let
$$
f_n(x)=\max\set{\min\set{h_n(x),2R_n-d(x,0)},0}
$$
for $x\in M$. Then $f_n\in B_{\Lip_0(M)}$ is positive, vanishes outside of $B(0,2R_n)$, and agrees with $h_n$ on $B(0,R_n)$, in particular on $\supp(m_n)$, therefore $\abs{\duality{f_n,m_n}}>\varepsilon/4$. It follows that, for every $x\in M$, $f_n(x)\neq 0$ for at most one value of $n$, therefore $\sum_nf_n$ is a WUC series in $\Lip_0(M)$ \cite[Lemma 6]{AP_normal}. But then, for every $n$
$$
\abs{\duality{f_n,\gamma_n}} \geq \abs{\duality{f_n,m_n}} - 2^{-n}\varepsilon > \frac{\varepsilon}{4}-2^{-n}\varepsilon
$$
contradicting the fact that $\Gamma$ is a (V$^*$)-set. This proves the Claim.

\medskip

Now we proceed with the main proof. By \cite[Theorem 3.2]{ANPP}, it is enough to show that for any $\varepsilon,\delta>0$ there exists a finite set $E\subset M$ such that $\Gamma\subset\lipfree{[E]_\delta}+\varepsilon B_{\lipfree{M}}$ (as before, we assume tacitly that $0\in E$). So fix $\varepsilon,\delta>0$. By the Claim, there is $R>0$ such that $\Gamma\subset\lipfree{B(0,R)}+\frac{\varepsilon}{5}B_{\lipfree{M}}$. Let $W=W_h^*:\lipfree{M}\to\lipfree{B(0,2R)}$ be the weighting operator given by \eqref{eq:weighting operator} for the function
$$
h(x) = \max\set{0,\min\set{1,2-\frac{1}{R}d(x,0)}} ,
$$
so that $\norm{W}\leq 3$ and $W$ is the identity on $\lipfree{B(0,R)}$. Then $W(\Gamma)$ is a (V$^*$)-set in $\lipfree{B(0,2R)}$ by Fact \ref{fact:sets}, and $\Gamma\subset W(\Gamma)+\frac{4}{5}\varepsilon B_{\lipfree{M}}$. Indeed, for every $\gamma\in\Gamma$ we can find $m\in\lipfree{B(0,R)}$ with $\norm{m-\gamma}\leq\frac{\varepsilon}{5}$ and thus
$$
\norm{\gamma-W\gamma} \leq \norm{\gamma-m} + \norm{m-Wm} + \norm{W(m-\gamma)} \leq 4\norm{m-\gamma} \leq \frac{4}{5}\varepsilon
$$
as $m=Wm$.

It now suffices to prove that there is a finite set $E\subset M$ such that $W(\Gamma)\subset\lipfree{[E]_\delta}+\frac{\varepsilon}{5} B_{\lipfree{M}}$. We will use a variation of the argument from the Claim to prove this. Assume otherwise, that is, that for every finite set $E\subset M$ there is $\gamma\in W(\Gamma)$ with $\dist(\gamma,\lipfree{[E]_\delta})>\frac{\varepsilon}{5}$. Pick $E_0=\set{0}$ and, iteratively, once $E_{n-1}$ has been chosen, pick
\begin{itemize}
\item an element $\gamma_n\in W(\Gamma)$ such that $\dist(\gamma_n,\lipfree{[E_{n-1}]_\delta})>\frac{\varepsilon}{5}$,
\item an element $m_n\in\lipfree{M}$ with finite support such that $\norm{m_n-\gamma_n}\leq 2^{-n}\frac{\varepsilon}{5}$,
\item and $E_n=E_{n-1}\cup\supp(m_n)$.
\end{itemize}
Note that $\dist(m_n,\lipfree{[E_{n-1}]_\delta})>\frac{\varepsilon}{5}-2^{-n}\frac{\varepsilon}{5}\geq\varepsilon/10$.
By the Hahn-Banach theorem, there exists $h_n\in B_{\Lip_0(M)}$ such that $\duality{h_n,m_n}>\varepsilon/10$ and $h_n$ vanishes on $[E_{n-1}]_\delta$. By replacing $h_n$ with either its positive or its negative part, we may assume that $h_n\geq 0$ and $\abs{\duality{h_n,m_n}}>\varepsilon/20$ instead. Now let
$$
f_n(x) = \max\set{0,\max\set{h_n(y)-\frac{4R}{\delta}d(x,y) \,:\, y\in\supp(m_n)}}
$$
for $x\in M$. That is, $f_n$ is the smallest positive $(4R/\delta)$-Lipschitz extension of $h_n$ from $\supp(m_n)$ to $M$. In particular, $0\leq f_n\leq h_n$ and so $f_n=0$ on $[E_{n-1}]_\delta$. Let us see that the $f_n$ have disjoint supports: if $k<n$ and $f_n(x)\neq 0$ then $x\notin [E_{n-1}]_\delta\supset [E_k]_\delta$. Therefore, for $y\in\supp(m_k)$ we have $d(x,y)>\delta$, thus
$$
h_k(y)-4R\delta^{-1}d(x,y)<4R-4R\delta^{-1}\delta=0
$$
and so $f_k(x)=0$. Since $\norm{f_n}\leq 4R/\delta$ for all $n$, it follows that $\sum_n f_n$ is a WUC series by \cite[Lemma 6]{AP_normal}. But we also have
$$
\abs{\duality{f_n,\gamma_n}} \geq \abs{\duality{f_n,m_n}} - 2^{-n}\frac{\varepsilon}{5}\cdot 4R\delta^{-1} > \frac{\varepsilon}{20} - 2^{-n}\tfrac{4}{5}\varepsilon R\delta^{-1}
$$
and therefore
$$
\limsup_{n\to\infty}\sup_{\gamma\in W(\Gamma)}\abs{\duality{f_n,\gamma}} \geq \frac{\varepsilon}{20}
$$
contradicting the fact that $W(\Gamma)$ is a (V$^*$)-set. This finishes the proof.
\end{proof}

The following is an immediate consequence of Theorem \ref{th:v*_tight}.

\begin{corollary}
(V$^*$)-sets in Lipschitz-free spaces are separable.
\end{corollary}

The following improvement of \cite[Theorem 2.11]{ANPP}, communicated to us by A. Proch\'azka, also follows easily. It shows that the property ``not containing $X$'' is compactly determined for a large class of Banach spaces $X$ including e.g.\ $C([0,1])$.

\begin{corollary}\label{cr:compact determination subspace}
Let $X$ be a Banach space that contains no complemented copy of $\ell_1$. If $\lipfree{M}$ contains an isomorphic copy of $X$, then there exists a compact set $K\subset M$ such that $\lipfree{K}$ contains an isomorphic copy of $X$.
\end{corollary}

\begin{proof}
Let $Y\subset\lipfree{M}$ be isomorphic to $X$. Note that $Y$ cannot contain a copy $Z$ of $\ell_1$ that is complemented in $\lipfree{M}$: if it did and $P$ was a projection of $\lipfree{M}$ onto $Z$, then $P\restrict_Y$ would witness that $Z$ is complemented in $Y$ as well. By Proposition \ref{pr:v* sets}, $B_Y$ is a (V$^*$)-set with respect to $\lipfree{M}$. Hence $B_Y$ is tight by Theorem \ref{th:v*_tight}. Thus, by \cite[Theorem 3.2]{ANPP}, there exist a compact $K\subset M$ and a sequence of bounded operators $T_n:\lipfree{M}\to\lipfree{M}$ that converge uniformly on $B_Y$ to a linear map $T:B_Y\to\lipfree{K}$ such that $\norm{y-Ty}\leq\frac{1}{2}$ for all $y\in B_Y$. Because $T_n\restrict_{B_Y}$ is continuous for each $n$, $T$ is also continuous, and thus can be extended to an operator $T:Y\to\lipfree{K}$. Since $\norm{Ty}\geq\frac{1}{2}$ for all $y\in S_Y$, $T$ is an isomorphism of $Y$ onto $T(Y)\subset\lipfree{K}$.
\end{proof}

We are now ready to prove local determination for property (V$^*$). In other words: $\lipfree{M}$ fails property (V$^*$) if and only if there exists a point $x\in M$ such that $\lipfree{B(x,r)}$ fails property (V$^*$) for every $r>0$.

\begin{customtheorem}{\ref{mth:locally determined}}
Property (V$^*$) is locally determined for Lipschitz-free spaces.
\end{customtheorem}

\begin{proof}
Let $\lipfree{M}$ be a Lipschitz-free space, and fix a (V$^*$)-set $\Gamma$ (with respect to $\lipfree{M}$). We wish to prove that, under appropriate hypotheses, $\Gamma$ is relatively weakly compact. By Grothendieck's lemma \cite[Lemma II.C.7]{Wojt} it is enough to show that, for any given $\varepsilon>0$, there exists a relatively weakly compact set $\Gamma_\varepsilon\subset\lipfree{M}$ such that $\Gamma\subset\Gamma_\varepsilon+\varepsilon B_{\lipfree{M}}$.

Theorem \ref{th:v*_tight} shows that $\Gamma$ is tight. Thus, by \cite[Theorem 3.2]{ANPP}, there exist a compact $K\subset M$, a sequence of bounded operators $T_k:\lipfree{M}\to\lipfree{M}$, and a linear mapping $T:\Gamma\to\lipfree{K}$ such that $\norm{\gamma-T\gamma}\leq\varepsilon$ for all $\gamma\in\Gamma$ and $T_k\gamma$ converges to $T\gamma$ uniformly on $\gamma\in\Gamma$. We may assume that $0\in K$. Note that $T_k(\Gamma)$ is a (V$^*$)-set with respect to $\lipfree{M}$ for every $k$ according to Fact \ref{fact:sets}. 

Now let $U$ be a neighborhood of $K$. We claim that $T(\Gamma)$ is a (V$^*$)-set with respect to $\lipfree{\cl{U}}$. Indeed, take a WUC series $\sum_nf_n$ in $\Lip_0(\cl{U})$, then $\Lip(f_n)$ is bounded, say by $C>0$.  Denote $r=d(K,M\setminus U)$ and consider the Lipschitz function $h:M\to\RR$ defined by
$$
h(p)=\max\{1-r^{-1}d(p,K),0\}
$$
for every $p\in M$.
Let $W=W_h:\Lip_0(\overline{U})\to\Lip_0(M)$ be the weighting operator associated to $h$ given by \eqref{eq:weighting operator}, then $\sum_n W(f_n)$ is a WUC series in $\Lip_0(M)$. Fix $\delta>0$, then choose $k$ such that $\norm{T_k\gamma-T\gamma}\leq\delta/C\norm{W}$ for all $\gamma\in\Gamma$ and, since $T_k(\Gamma)$ is a (V$^*$)-set with respect to $\lipfree{M}$, further choose $N$ such that $\abs{\duality{W(f_n),T_k\gamma}}\leq\delta$ for all $n\geq N$ and $\gamma\in\Gamma$. Then, for such $n$ and $\gamma$ we have
\begin{align*}
\abs{\duality{W(f_n),T\gamma}} &\leq \abs{\duality{W(f_n),T_k\gamma}}+\abs{\duality{W(f_n),T\gamma-T_k\gamma}} \\
&\leq \delta+C\norm{W}\frac{\delta}{C\norm{W}} = 2\delta.
\end{align*}
Because $\supp(T\gamma)\subset K$ for every $\gamma\in\Gamma$ and $W(f_n)\restrict_K=f_n\restrict_K$, we have $\abs{\duality{f_n,T\gamma}}=\abs{\duality{W(f_n),T\gamma}}\leq 2\delta$. So, $T(\Gamma)$ is a (V$^*$)-set with respect to $\lipfree{\cl{U}}$ as claimed.

Finally, we apply our hypotheses. Suppose that every $x\in M$ has a neighborhood $U$ such that $\lipfree{\cl{U}}$ has property (V$^*$). Then by Lemma \ref{lm:local to compact} we may choose a neighborhood $U$ of $K$ such that $\lipfree{\cl{U}}$ has property (V$^*$). Since $T(\Gamma)$ is a (V$^*$)-set with respect to $\lipfree{\cl{U}}$, it must be relatively weakly compact, and so we take $\Gamma_\varepsilon=T(\Gamma)$ to finish the proof.
\end{proof}

When $M$ is discrete, each singleton is a neighborhood of itself, so the following particular case is immediate.

\begin{corollary}
\label{cr:discrete}
If $M$ is discrete, then $\lipfree{M}$ has property (V$^*$).
\end{corollary}

\subsection{Sufficient conditions for compact determination}

While we cannot establish whether property (V$^*$) is compactly determined in general, we can pinpoint several situations where it is. For instance, by Proposition \ref{pr:local vs compact determination}, property (V$^*$) is compactly determined for $\lipfree{M}$ when $M$ is locally compact. One case where we can establish property (V$^*$) for $\lipfree{K}$, with $K$ compact, is when $M$ (and hence also $K$) is purely 1-unrectifiable. Indeed, if $K$ is compact and purely 1-unrectifiable then $\lipfree{K}$ is L-embedded \cite[Theorem 6.5 and Proposition 6.2]{APS2}, so it has property (V$^*$) by a result of Pfitzner \cite{Pfitzner_v*}.
Alternatively, $\lipfree{K}$ is the dual of an M-embedded Banach space (see \cite[Theorem 3.2 and the start of Section 3.1]{AGPP}) and thus has property (V$^*$) by the results in \cite[p. 129]{HWW}. Thus we get the following strenghtening of Corollary \ref{cr:discrete}.

\begin{customtheorem}{\ref{mth:locally compact p1u}}
If $M$ is locally compact and purely 1-unrectifiable, then $\lipfree{M}$ has property (V$^*$).
\end{customtheorem}

Besides discrete spaces, this holds when $M$ is a snowflake of a locally compact metric space.

The following proposition collects other technical conditions that are sufficient for compact determination of property (V$^*$).

\begin{proposition}\label{pr:sufficient conditions}
Suppose that $M$ satisfies at least one of the following conditions:
\begin{enumerate}[label={\upshape{(\alph*)}}]
\item\label{diamond} For every compact set $K\subset M$ and subset $\Gamma\subset\lipfree{K}$, the following holds:\\ if $\Gamma$ is a (V$^*$)-set with respect to $\lipfree{\cl{U}}$ for every open set $U\supset K$, then it is also a (V$^*$)-set with respect to $\lipfree{L}$ for some compact $L\supset K$.
\item\label{wuc extension} For every compact $K\subset M$ and every WUC series $\sum_nf_n$ in $\Lip_0(K)$ there exists a WUC series $\sum_ng_n$ in $\Lip_0(M)$ such that $g_n\restrict_K=f_n$.
\item\label{l1 extension} For every compact $K\subset M$, every Lipschitz mapping $\varphi:K\to\ell_1$ can be extended to a Lipschitz mapping $\Phi:M\to\ell_1$ such that $\Phi\restrict_K=\varphi$.
\item\label{extension} For every compact $K\subset M$ there exists a linear extension operator $E:\Lip_0(K)\to\Lip_0(M)$ such that $E(f)\restrict_K=f$ for all $f\in\Lip_0(K)$.
\item\label{retract} Every compact subset of $M$ is contained in a compact Lipschitz retract of $M$.
\end{enumerate}
Then property (V$^*$) is compactly determined for $\lipfree{M}$.
\end{proposition}

\begin{proof}
Suppose that $\lipfree{K}$ has property (V$^*$) for every compact $K\subset M$. Assume first that \ref{diamond} holds. We follow the proof of Theorem \ref{mth:locally determined} up until the last paragraph. At that point, we have that $T(\Gamma)\subset\lipfree{K}$ is a (V$^*$)-set with respect to $\lipfree{\cl{U}}$ for every open $U\supset K$. By condition \ref{diamond}, we deduce the existence of a compact $L\supset K$ such that $T(\Gamma)$ is a (V$^*$)-set with respect to $\lipfree{L}$. Since $\lipfree{L}$ has property (V$^*$), $T(\Gamma)$ is relatively weakly compact. The proof that $\lipfree{M}$ has property (V$^*$) is now finished as in Theorem \ref{mth:locally determined}.

Next, we show that \ref{wuc extension} implies \ref{diamond}. Suppose that \ref{wuc extension} holds and $K\subset M$ is compact. Let $\sum_nf_n$ be a WUC series in $\Lip_0(K)$ and extend it to a WUC series $\sum_ng_n$ in $\Lip_0(M)$. Given any $\gamma\in\lipfree{K}$, we have $\duality{g_n,\gamma}=\duality{f_n,\gamma}$ as $f_n$ and $g_n$ agree on $K$. Thus, if $\Gamma\subset\lipfree{K}$ is a (V$^*$)-set with respect to $\lipfree{M}$, then
$$
\lim_{n\to\infty}\sup_{\gamma\in\Gamma}\abs{\duality{f_n,\gamma}} = \lim_{n\to\infty}\sup_{\gamma\in\Gamma}\abs{\duality{g_n,\gamma}} = 0
$$
so $\Gamma$ is a (V$^*$)-set with respect to $\lipfree{K}$ as well, and \ref{diamond} holds with $L=K$.

Notice that \ref{wuc extension} and \ref{l1 extension} are equivalent. Indeed, by Fact \ref{fact:wuc}\ref{wuc l1}, WUC series in $\Lip_0(M)$ correspond exactly to operators $T:\lipfree{M}\to\ell_1$, and these, in turn, correspond exactly to Lipschitz mappings $\varphi:M\to\ell_1$ (such that $\varphi(0)=0$) by the linearization property of Lipschitz-free spaces.

If \ref{extension} holds and $\sum_nf_n$ is a WUC series in $\Lip_0(K)$ then $\sum_n E(f_n)$ is a WUC series in $\Lip_0(M)$ by Fact \ref{fact:wuc}\ref{wuc operator}, so \ref{wuc extension} holds.

Finally, we show that \ref{retract} implies \ref{diamond}. Assume that \ref{retract} holds and let $K\subset M$ be compact, then there exist a compact $L\supset K\cup\set{0}$ and a Lipschitz retraction $\pi:M\to L$. Its linearization $\widehat{\pi}:\lipfree{M}\to\lipfree{L}$ is then a projection onto $\lipfree{L}$, as $\widehat{\pi}(\delta(x))=\delta(\pi(x))=\delta(x)$ for $x\in L$. Thus, if $\Gamma\subset\lipfree{K}$ is a (V$^*$)-set with respect to $\lipfree{M}$, then $\Gamma=\widehat{\pi}(\Gamma)$ is a (V$^*$)-set with respect to $\lipfree{L}$ by Fact \ref{fact:sets}. Thus \ref{diamond} holds.
\end{proof}

It is conceivable that condition \ref{diamond} holds for every metric space $M$, even with $L=K$, but we do not know if this is true. If so, this would imply that property (V$^*$) is always compactly determined.

On the other hand, condition \ref{l1 extension} (and hence \ref{wuc extension}) does not hold in general, it fails e.g.\ for $M=\ell_1$. Indeed, by \cite[Corollary D.2]{Makarychev}, there exist a sequence $(F_n)$ of finite subsets of $\ell_1$ and $1$-Lipschitz mappings $\varphi_n:F_n\to\ell_1$ that admit no extension $\Phi_n:\ell_1\to\ell_1$ with $\Lip(\Phi_n)\leq n$. It is then straightforward to combine them into a compact $K\subset\ell_1$ and a Lipschitz mapping $\varphi:K\to\ell_1$ that admits no Lipschitz extension $\Phi:\ell_1\to\ell_1$.

Consequently, extension operators as in \ref{extension} do not always exist. However, using the extension results due to Naor and Silberman \cite{NS}, it is shown in \cite[Lemma 3.2]{FG} that such an operator exists whenever $K\subset M$ has finite Nagata dimension. Therefore, property (V$^*$) is compactly determined for any $M$ whose compact subsets have finite Nagata dimension, e.g.\ doubling spaces or ultrametric spaces. We refer the reader to \cite{LS} for the definition and basics of Nagata dimension, and to \cite{AmPu,BB,LN} for more related work on extension operators.

Another suitable application of compact determination of $\lipfree{M}$ is when $M$ is (a subset of) a superreflexive Banach space, as it is known that $\lipfree{K}$ then has property (V$^*$) for all compact $K\subset M$ \cite[Corollary 11]{KP}. Local compactness would then imply property (V$^*$) for $\lipfree{M}$ by Proposition \ref{pr:local vs compact determination}, however superreflexive spaces are of course locally compact only when they are finite-dimensional, in which case they are proper and \cite{KP} already yields the full result. In the infinite-dimensional case, we can use condition \ref{retract} to solve the Hilbert case.

\begin{customtheorem}{\ref{mth:hilbert v*}}
If $X$ is a Hilbert space then $\lipfree{X}$ has property (V$^*$).
\end{customtheorem}

\begin{proof}
Note that condition \ref{retract} holds in Hilbert spaces: if $K\subset X$ is compact then $L=\cl{\conv}\,K$ is compact and convex, so there exists a non-expansive retraction $\pi\colon X\to L$, namely the nearest point projection. Thus property (V$^*$) is compactly determined for $\lipfree{X}$ by Proposition \ref{pr:sufficient conditions}. By \cite[Corollary 11]{KP}, $\lipfree{K}$ has property (V$^*$) for every compact $K\subset X$, so the theorem follows.
\end{proof}

Kopeck\'a proved in \cite[Proposition 3.3]{Kopecka07} that property \ref{retract} does not hold in the spaces $\ell_p$, $p\in (1,\infty)\setminus\set{2}$ for 1-Lipschitz retracts. We do not know whether it holds for Lipschitz retracts without a restriction on the Lipschitz constant; if it does, that would be enough to conclude that $\lipfree{\ell_p}$ has property (V$^*$).

\begin{remark}
During the review process of this paper, William B. Johnson communicated privately that property \ref{retract} holds in $\ell_p$ spaces for $p \in [1,\infty)$. From this result, Proposition \ref{pr:sufficient conditions} and \cite[Corollary 11]{KP}, it would follow that $\mathcal{F}(\ell_p)$ has property (V$^*$) for $p \in (1,\infty)$.
\end{remark}

%%%%%%%%%%%%%%%%%%%%%%%%%%%%%%%%%%%%%%%%%%%%%%%
\section{Lipschitz-free spaces over Carnot-Carath\'eodory spaces}

In this section, we provide another family of metric spaces whose Lipschitz-free spaces have property (V$^*$), namely certain Carnot-Carath\'eodory spaces. In \cite{Bour1,Bour2}, Bourgain identified a special class of subspaces of $C(K,E)$ -- the space of continuous functions on a compact Hausdorff space $K$ with values in a finite-dimensional Banach space $E$ -- whose duals have property (V$^*$). An important particular case is that $C^1([0,1]^n)^*$ has property (V$^*$) (see \cite[Theorem III.D.31 and Examples III.D.30]{Wojt}). Bourgain's argument was adapted to Lipschitz-free spaces in \cite{KP}, where it was shown that Lipschitz-free spaces over compact subsets of superreflexive Banach spaces have property (V$^*$). One of the properties of superreflexive spaces that was key to enabling such adaptation is the existence of uniform approximations of Lipschitz functions on a superreflexive space $X$ by $C^1$-smooth functions preserving the Lipschitz constant up to a multiplicative constant depending only on the space $X$  \cite[Corollary 8]{HJ}. If we want to apply similar methods in other metric spaces, suitable candidates therefore appear to be doubling metric measure spaces supporting a Poincar\'e inequality that have a rich curve structure and admit approximations of Lipschitz functions by smooth Lipschitz functions with continuous upper gradients and controlled Lipschitz constants. We will specify this setting in the next section.

\subsection{Carnot--Carath\'eodory spaces and upper gradients}

The main references for the information presented in this section are \cite{HK,HKST}.

Recall that given any metric space $M$, a continuous map $\gamma\colon[a,b]\to M$, where $a\leq b\in \RR$, is called a \textit{curve} in $M$. If $p,q\in M$ and $\gamma\colon [a,b]\to M$ is a 1-Lipschitz curve such that $\gamma(a)=p$, $\gamma(b)=q$, and $b-a=d(p,q)$, then $\gamma$ is actually an isometric embedding of $[a,b]$ into $M$. The mapping $\gamma$ is then called a \emph{geodesic (curve)}
joining $p$ and $q$. A metric space is termed \emph{geodesic} if every two points can be joined by a geodesic.

Fix $d,s\in\NN$. Let $G\subset(\RR^d,\enorm{\cdot}{d})$ be an open connected set and let $X=(X_1,\ldots,X_s)$ be a family of vector fields on $G$ with real locally Lipschitz continuous coefficients. We say that a curve $\gamma\colon[a,b]\to G$ is \textit{admissible} (with respect to $X$) if it is absolutely continuous and there exist Borel functions $c_1,\ldots,c_s\colon[a,b]\to\RR$ such that $\sum_{j=1}^s c_j(t)^2\leq 1$ for all $t\in[a,b]$ and
$$\gamma'(t)=\sum_{j=1}^s c_j(t)X_j(\gamma(t))$$
for almost every $t\in[a,b]$.

\begin{definition}
     Let $G$ and $X$ be as above. Then for every $x,y\in G$ we define $d_{cc}(x,y)$ to be the infimum of $b\geq 0$ such that there exists an admissible curve $\gamma\colon[0,b]\to G$ satisfying $\gamma(0)=x$ and $\gamma(b)=y$, and $d_{cc}(x,y)=\infty$ if no such admissible curve exists. The function $d_{cc}$ is called the \textit{Carnot-Carath\'eodory distance} on $G$ associated with $X$.
\end{definition}
Note that if $d_{cc}(x,y)$ is finite for every $x,y\in G$, i.e.\ every two points in $G$ are connected by an admissible curve, then $d_{cc}$ is a metric on $G$ called the Carnot-Carath\'eodory metric (see e.g.\ \cite[Proposition 11.2]{HK}). Such pair $(G,d_{cc})$ is referred to as \textit{Carnot-Carath\'eodory metric space}. By \cite[Proposition 11.2]{HK} again, we also have that the identity $(G,d_{cc})\to(G,\enorm{\cdot}{d})$ is continuous. Every admissible curve in $G$ is clearly 1-Lipschitz with respect to $d_{cc}$, hence rectifiable and compact, and admits an arc-length parametrization that is again 1-Lipschitz with respect to $d_{cc}$ \cite[Proposition 5.1.8]{HKST} and therefore also admissible \cite[Proposition 11.4]{HK}.
It follows that the distance between two points in $(G,d_{cc})$ equals the infimum of the lengths of the admissible curves joining them, so $(G,d_{cc})$ is a length space \cite[Proposition 11.5]{HK}.

Another important feature of Carnot-Carath\'eodory metric spaces is that they admit a linear differential structure. Indeed, we first recall the notion of an upper gradient.

\begin{definition}
Let $M$ be a metric space, $\Omega\subset M$ an open subset, and $f\colon \Omega\to\RR$ an arbitrary function. We say that a Borel function $g\colon \Omega\to[0,\infty]$ is an \textit{upper gradient} of the function $f$ on $\Omega$ if for every $1$-Lipschitz curve $\gamma\colon[a,b]\to\Omega$ we have
$$|f(\gamma(b))-f(\gamma(a))|\leq\int_a^bg(\gamma(t))\,dt.$$
\end{definition}
If $X=(X_1,\ldots,X_s)$ is the family of vector fields on $G$ defining the Carnot-Carath\'eodory metric $d_{cc}$ on $G$, and $f\in C^{\infty}(\Omega,\enorm{\cdot}{d})$ for some open $\Omega\subset G$, then we define the action of each vector field $X_j$ on $f$ as
$$X_j(f)(x)=\duality{X_j(x),\nabla f(x)}$$
for every $1\leq j\leq s$ and $x\in\Omega$. We denote 
\begin{equation}
    \label{eq:definition_gradient}
X(f)=(X_1(f),\ldots,X_s(f)),
\end{equation}
so $X(f)\colon\Omega\to\RR^s$. Note that the mapping $f\mapsto X(f)$ is linear and satisfies
\begin{equation}
    \label{eq:product_rule}
    X(fg)=fX(g)+X(f)g
\end{equation}
for every $f,g\in C^{\infty}(\Omega,\enorm{\cdot}{d})$. Moreover, we have the following.

\begin{proposition}[{\cite[Proposition 11.6]{HK}}]
\label{prop:upper_gradient}
Let $(G,d_{cc})$ be a Carnot-Carath\'eodory metric space and $\Omega\subset G$ an open subset. If $X$ is the family of vector fields defining $d_{cc}$ and $f\in C^\infty(\Omega,\enorm{\cdot}{d})$, then $\enorm{X(f)}{s}$ is an upper gradient of $f$ on $(\Omega,d_{cc})$.
\end{proposition}

We will be interested in Carnot-Carath\'eodory metric spaces $(G,d_{cc})$ for which the following additional condition is true:
\begin{itemize}
    \item[$(\diamond)\,\,$]
    \label{eq:condition_heart}
        there exists an $\alpha\in(0,1)$ such that for every bounded set $K\subset G$ there are constants $C_1, C_2>0$ satisfying
$$
C_1\enorm{x-y}{d}\leq d_{cc}(x,y)\leq C_2\enorm{x-y}{d}^\alpha
$$
whenever $x,y\in K$.
\end{itemize}
Note that for a Carnot-Carath\'eodory metric space $(G,d_{cc})$ satisfying $(\diamond)$ the identity $(G,d_{cc})\to (G,\enorm{\cdot}{d})$ is a homeomorphism. Moreover, $(G,d_{cc})$ is proper, and hence also complete. Thus, as a proper length space, it is geodesic \cite[Lemma 8.3.11]{HKST}.

A prime example of such spaces are \emph{Carnot groups} (in particular, the Heisenberg group) with the Carnot-Carath\'eodory metric. A Carnot group $G$ is a connected, simply connected Lie group whose associated Lie algebra $\mathfrak{g}$ admits a stratification, i.e.\ a direct sum decomposition $\mathfrak{g}=V_1\oplus V_2\oplus\dots\oplus V_n$ such that $[V_1,V_i]=V_{i+1}$ if $1\leq i < n$ and $[V_1,V_n]=\{0\}$. Being nilpotent, $G$ is then diffeomorphic to $\RR^d$ for some $d\in\N$. Let $X_1, X_2,\ldots, X_s$ be a basis of $V_1$. Then each $X_i$ can be uniquely extended to a left-invariant vector field on $G$, which we denote $X_i$ again, using the left translations on $G$. The Carnot-Carath\'eodory distance $d_{cc}$ on $G$ associated with the family of vector fields $X=(X_1,\dots, X_s)$ is a metric on $G$ that satisfies condition $(\diamond)$ with $\alpha=1/n$ (see e.g.\ Propositions 11.14 and 11.15 and equation (64) in \cite{HK}). For unexplained notions and more details, we refer the reader to \cite{LD}.

In Lemma \ref{lm:approximation} below, we prove an approximation result that will be crucial for showing that the Lipschitz-free spaces over Carnot-Carath\'eodory metric spaces satisfying $(\diamond)$ have property (V$^*$). So, suppose $(G,d_{cc})$ is a Carnot-Carath\'eodory metric space satisfying $(\diamond)$. From now on, Lipschitz constants and the terminology Lipschitz, length, etc.\ are meant with respect to $d_{cc}$. If we want to consider the Euclidean norm instead, we specify it by explicitly writing $\enorm{\cdot}{d}$. Also, all admissible curves are assumed to be parametrized by arc length. Given $p,q\in G$, we will denote $\Gamma_{pq}\subset G^{[0,d_{cc}(p,q)]}$ the set of all geodesic curves joining $p$ and $q$. Note that $\Gamma_{pq}$ is non-empty for every $p,q\in G$ as the metric space $(G,d_{cc})$ is geodesic, and that every geodesic curve in $(G,d_{cc})$ is admissible by \cite[Proposition 11.4]{HK}. Finally, for a compact set $K\subset G$, set
$$
     \Gamma_K:=\overline{\bigcup\left\{\gamma([0,d_{cc}(p,q)]) \,\colon\, p\neq q\in  K, \gamma \in \Gamma_{pq}\right\}}\,\subset\, G.
    $$
Observe that $ K\subset  \Gamma_K\subset G$ and, since $ K$ is bounded, the lengths $d_{cc}(p,q)$ of curves $\gamma$ have a uniform bound for all distinct $p$ and $q$ in $ K$. Hence, using the fact that $G$ is proper, $ \Gamma_K$ is also compact.

\begin{lemma}
\label{lm:approximation}
 Let $(G,d_{cc})$ be a Carnot-Carath\'eodory metric space satisfying $(\diamond)$ and let $x_0\in G$ be arbitrary. Then there exist a compact neighborhood $K\subset G$ of $x_0$, an open connected set $\Omega\subset G$ such that $\Gamma_K\subset \Omega$, and a constant $\beta>0$ such that for every Lipschitz function $f\colon G\to \RR$ and every $\varepsilon>0$, there exists a function $g\colon\Omega\to\RR$ satisfying the following:
  \begin{enumerate}[label={\upshape{(\roman*)}}]
  \item \label{g_smooth}
  $g\in C^\infty(\Omega,\enorm{\cdot}{d})$,
 \item \label{g_approximation} $\sup_{x\in  \Gamma_K} |f(x)-g(x)|<\varepsilon$,
 \item \label{g_gradient} $\sup_{x\in  \Gamma_K}\enorm{X(g)(x)}{s}\leq \Lip(f)$ (where $X(g)$ is as in \eqref{eq:definition_gradient}),
  \item \label{g_value0} $g(x_0)=f(x_0)$,
 \item \label{g_holder}
 $|g(x)-g(y)|\leq \Lip(f)\beta \enorm{x-y}{d}^\alpha$ for every $x\neq y\in  \Gamma_K$ (where $\alpha$ is from condition $(\diamond)$).
 \end{enumerate}
 In particular, 
 the function $g\restrict_K\colon K\to\RR$ is Lipschitz with $\Lip(g\restrict_K)\leq \Lip(f)$.
 \end{lemma}

\begin{proof}
Throughout this proof, we will use the notation $B_{\ell_2^d}(x,r)$ and $B_{cc}(x,r)$ to refer to closed balls with respect to the Euclidean and Carnot-Carath\'eodory metrics on $G$, respectively, and $B^o_{\ell_2^d}(x,r)$ for open balls in the Euclidean metric.

Let $x_0\in G$ and $R>0$ such that $B^o_{\ell_2^d}(x_0,2R)\subset G$. Because $(G,d_{cc})$ is homeomorphic to $(G,\enorm{\cdot}{d})$, we can find an $r>0$ such that $B_{cc}(x_0,3r)\subset B^o_{\ell_2^d}(x_0,R)$. Put $$K:=B_{cc}(x_0,r)\qquad \textup{and} \qquad \Omega:=B^o_{\ell_2^d}(x_0,R).$$ Then $\Gamma_K\subset\Omega$ because the length of each geodesic joining two points from $K$ is at most $2r$.
   
    Take an arbitrary $f\colon G\to \RR$ Lipschitz and $\varepsilon>0$. Denote $L=\Lip(f)<\infty$. Find $m_0\in\NN$ such that $\frac{1}{m_0}<R$. For $m\in \N$, $m\geq m_0$, consider the standard mollifier defined on $\RR^d$ by $u_m(x)=m^d\varphi(mx)$, where $\varphi\in C_0^\infty(\RR^d,\enorm{\cdot}{d})$ is a non-negative function with $\supp(\varphi)\subset B^o_{\ell_2^d}(0,1)$ and $\int_{\RR^d} \varphi=1$. Then $f\ast u_m\in C^\infty(\Omega,\enorm{\cdot}{d})$ and, by \cite[Proposition 11.10]{HK}, we have that $f\ast u_m\to f$ as $m\to\infty$ uniformly on compact subsets of $\Omega$ and
    $$
    \|X(f\ast u_m)(x)\|_{\ell_2^s}\leq L+\|A_{1/m}f(x)\|_{\ell_2^s}
    $$
    for every $x\in \Omega$, where $\|A_{1/m}f\|_{\ell_2^s}\to 0$ as $m\to\infty$ uniformly on compact sets. Let $\widetilde{\varepsilon}>0$ be so small that
    $$ \frac{L \widetilde{\varepsilon}}{L+\widetilde{\varepsilon}}\pare{\sup_{x\in \Gamma_K}d_{cc}(x,x_0)+2}<\varepsilon. $$
    Then we can find $\widetilde{m}_0>m_0$ such that
    \begin{equation}
        \label{eq:convolution}
        |(f\ast u_{\widetilde{m}_0})(x)-f(x)|<\widetilde{\varepsilon} \quad \text{and} \quad \|A_{1/\widetilde{m}_0}f(x)\|_{\ell_2^s}<\widetilde{\varepsilon}
    \end{equation}
    for every $x\in  \Gamma_K$. Moreover, since $ \Gamma_K\subset \Omega$,
    \begin{equation}
        \label{eq:gradient of convolution}
        \|X(f\ast u_{\widetilde{m}_0})(x)\|_{\ell_2^s}\leq L+\widetilde{\varepsilon}
    \end{equation}
    for every $x\in  \Gamma_K$. Set 
    $$\widetilde{g}=\frac{L}{L+\widetilde{\varepsilon}}\cdot \left(f\ast u_{\widetilde{m}_0}\right)\quad\textup{and}\quad g=\widetilde{g}+(f(x_0)-\widetilde{g}(x_0))\mathbbm{1}_\Omega.
    $$
    Then $g\in C^\infty(\Omega,\enorm{\cdot}{d})$ as desired in \ref{g_smooth}, and, using \eqref{eq:convolution}, \eqref{eq:gradient of convolution} and the fact that $f$ is $L$-Lipschitz, it is straightforward to verify that $g$ satisfies properties \ref{g_approximation}--\ref{g_value0}. To see \ref{g_holder}, let $x,y\in  \Gamma_K$. Since $f$ is $L$-Lipschitz, we have
    \begin{align*}
    |f\ast u_{\widetilde{m}_0}(x)-f\ast u_{\widetilde{m}_0}(y)|&\leq\int_{B^o_{\ell_2^d}(0,1/\widetilde{m}_0)}|f(x-z)-f(y-z)|u_{\widetilde{m}_0}(z)\,dz\\
       &\leq L d_{cc}(x-z,y-z)\leq L\beta\enorm{x-y}{d}^\alpha,
    \end{align*}
    where $\beta>0$ is the constant $C_2$ in $(\diamond)$ corresponding to the compact set $[ \Gamma_K]_{1/m_0}$.
    
    Now, using that $\|X(g)\|_{\ell_2^s}$ is an upper gradient of $g\in C^\infty(\Omega,\enorm{\cdot}{d})$ on $(\Omega,d_{cc})$ by Proposition \ref{prop:upper_gradient}, we obtain that for any $p\neq q \in  K$
    $$
    |g(p)-g(q)|\leq \int_0^{d_{cc}(p,q)} \|X(g)(\gamma_{pq}(t))\|_{\ell_2^s} \, dt \leq L d_{cc}(p,q),
    $$
    where $\gamma_{pq}\in \Gamma_{pq}$ is any geodesic curve joining $p$ and $q$. Indeed, the image of such a curve is contained in $ \Gamma_K$ and the estimate follows from \ref{g_gradient}. Therefore, $g\restrict_K$ is Lipschitz with $\Lip(g\restrict_K)\leq L$, which completes the proof.
\end{proof} 

\subsection{Proof of property (V*) in \texorpdfstring{$\lipfree{G}$}{F(G)}}

This section is devoted to proving property (V$^*$) for Lipschitz-free spaces over Carnot-Carath\'eodory spaces satisfying $(\diamond)$.

In addition to the approximation Lemma \ref{lm:approximation}, another essential ingredient will be the following fact about Hilbert spaces that can be found for instance in \cite{Wojt}:

\begin{lemma}[{\cite[Chapter III.D, Lemma 32]{Wojt}}]
\label{lem:Hilbert_conv_comb}
Let $(H,\norm{\cdot})$ be a Hilbert space, $n\in\N$, and $x_1,x_2,\ldots,x_n\in B_H$. Then there exist nonempty sets $A,B\subset\{1,2,\ldots,n\}$ such that $\max A<\min B$ and 
$$
\norm{\frac{1}{|A|}\sum_{i\in A}x_i-\frac{1}{|B|}\sum_{i\in B}x_i}\leq \frac{4}{\sqrt{\log n}}.
$$
\end{lemma}
We also recall an observation from \cite{Wojt} to be used repeatedly in the proof of Theorem \ref{thm:carnot_compact}.
\begin{fact}[{\cite[observation on pp. 170--171]{Wojt}}]
\label{ob:index_set}
Any finite partition of the index set 
\begin{equation*}
\mathcal{J}=\{(n,j):n\in\N, 1\leq j\leq n\}
\end{equation*}
contains a set that has a subset with the same structure, i.e. of the form
$$\{(n_i,j^i_s):i\in\N, n_1<n_2<\ldots\in\N, s\in\{1,\ldots,i\},1\leq j^i_1<j^i_2<\ldots<j^i_i\leq n_i\}.$$
\end{fact}
In what follows, we will often consider extensions of Lipschitz functions from a subset of a metric space to the whole metric space preserving the Lipschitz constant. If not specified otherwise, we mean any such extension (existing e.g.\ due to McShane's theorem \cite[Theorem 1.33]{Weaver2}).

The proof of Theorem \ref{thm:carnot_compact} below follows Bourgain's proof of property (V$^*$) for the duals of so-called rich subspaces of $C(K,E)$-spaces as presented in \cite[Theorem III.D.31]{Wojt}, or more precisely, its modification for Lipschitz-free spaces from \cite{KP}.

\begin{theorem}
\label{thm:carnot_compact}
If $(G,d_{cc})$ is a Carnot-Carath\'eodory metric space satisfying $(\diamond)$, then $\lipfree{G}$ has properties (V$^*$) and (R).
\end{theorem}

Theorem \ref{mth:carnot v*} follows as a particular case.

\begin{proof}
By Theorem \ref{mth:locally determined}, property (V$^*$) is locally determined. Property (R) is compactly determined by \cite[Proposition 3.5]{Silber}, hence also locally determined by Proposition \ref{pr:local vs compact determination}. Thus, it is enough to prove that every point in $G$ has a compact neighborhood whose Lipschitz-free space has properties (V$^*$) and (R). We will show this for the neighborhoods obtained from Lemma \ref{lm:approximation}. 

So, let $x_0\in G$. Fix a compact neighborhood $K\subset G$ of $x_0$, an open connected set $\Omega$ satisfying $\Gamma_K\subset\Omega\subset G$, and a constant $\beta>0$, that we obtain for the given $x_0$ from Lemma \ref{lm:approximation}. Without loss of generality, we may choose $x_0$ as the base point of $K$. Suppose $\Gamma\subset\lipfree{K}$ is a bounded set that is not RSWC. Then neither is $\mathcal{K}:=\cl{\conv}\,\Gamma$, by Fact \ref{fact:sets}. We will find a WUC series $\sum_n h_n$ in $\Lip_0(K)$ such that 
$\inf_{n}\sup_{\kappa\in\mathcal{K}}\abs{\duality{h_n,\kappa}}>0$, proving thus that $\mathcal{K}$ is not a (V$^*$)-set, therefore neither is $\Gamma$ by Fact \ref{fact:sets}. This will show that $\lipfree{K}$ has properties (V$^*$) and (R), as required. Since the construction of $h_n$ is rather long, the proof will be split into several sections for readability.

\medskip
\textbf{Setup.}
Set
$$
R:=\sup_{x\in \Gamma_K}\norm{x}_{\ell_2^d}<\infty .
$$
Let $\mathcal{A}$ be the set of all functions $f:\Omega\to\RR$ satisfying these conditions:
\begin{enumerate}[label=(\alph*),itemsep=0.1cm]
\item \label{smoothness_of_f}
$f\in C^{\infty}\left(\Omega,\enorm{\cdot}{d}\right)$,
\item \label{norm_of_X(f)}
$\sup_{x\in \Gamma_K}\enorm{X(f)(x)}{s}\leq 1$ (where $X(f)$ is as in \eqref{eq:definition_gradient}),
\item \label{Lipschitzness_of_f}
$f\restrict_K\in B_{\Lip_0(K)}$,
\item \label{Holderness_of_f} $|f(x)-f(y)|\leq \beta \enorm{x-y}{d}^\alpha$ for every $x\neq y\in \Gamma_K$ (where $\alpha$ is from condition $(\diamond)$).
\end{enumerate} 
Note that $\mathcal{A}$ is a convex and symmetric set.

First, we show that there exist constants $C\geq 1$ and $\xi>0$ and, for each $n\in\N$ and $1\leq j\leq n$, functions $f_j^n\in\mathcal{A}$ and elements $\mu_j^n\in\aspan\delta(K)\subset\free{K}$ such that the following conditions are satisfied:
\begin{enumerate}[label=(\alph*),itemsep=0.1cm]
\setcounter{enumi}{4}
\item \label{norm_of_mu}
$\n{\mu_j^n}_{\lipfree{K}}\leq C,$
\item \label{approximation}
$\dist\left(\mu_j^n,\mathcal{K}\right)\leq \frac {\xi C}{2\xi+12C}$,
\item \label{evaluations} 
$\left|\ip{f_k^n}{\mu_j^n}\right|\leq\frac\xi 3\textup{ if }1\leq j<k\leq n \textup{ and }
\ip{f_k^n}{\mu_j^n}\geq\xi\textup{ if }\, 1\leq k\leq j\leq n.$
\end{enumerate} 
Indeed, set 
$$
C:=1+\sup\set{\n{\mu}_{\lipfree{K}}\,:\,\mu\in\mathcal{K}} .
$$
By Proposition \ref{pr:lancien_raja} there exist $\xi>0$ and, for each $n\in\NN$ and $1\leq j\leq n$, elements $m_j\in\mathcal{K}$ and functions $f_j\in B_{\Lip_0(K)}$ such that $\duality{f_k,m_j}=0$ if $j<k$ and $\duality{f_k,m_j}>\xi$ if $k\leq j$. We may approximate each $m_j$ in norm by some $\mu_j^n\in\lspan\,\delta(K)$ so that \ref{norm_of_mu} and \ref{approximation} hold and we still have $\abs{\duality{f_k,\mu_j^n}}<\frac{\xi}{3}$ for $j<k$ and $\duality{f_k,\mu_j^n}>\xi$ for $k\leq j$. Finally, we replace the functions $f_k$ with smooth approximations satisfying \ref{smoothness_of_f}--\ref{Holderness_of_f} as follows. Because the functionals $\mu_1^n,\ldots,\mu_n^n$ are finitely supported in the compact set $K$, we can apply Lemma \ref{lm:approximation} to arbitrary extensions of the functions $f_k$ to the entire space $G$ retaining the Lipschitz constants so that \ref{evaluations} is satisfied. We denote the resulting smooth functions $f_k^n$ for all $1\leq k\leq n$. All of them belong to $\mathcal{A}$ by the corresponding conditions in Lemma \ref{lm:approximation}. The construction of the two sequences is thus complete.

Put 
$$\mathfrak{M}:=\set{\mu_j^n:n\in\N, 1\leq j\leq n}.$$
We will define some objects for each $\mu_j^n\in\mathfrak{M}$; although they depend on the index $(n,j)$, we will often omit it from their notation for better readability. We trust that this will cause no confusion. So, for each $n\in\NN, 1\leq j\leq n$, fix an optimal representation of $\mu_j^n$ as in \eqref{eq:sum of molecules}, i.e.
\begin{equation}
\label{eq:representation of mu}
\mu_j^n=\sum_{i=1}^I a_i m_{p_i q_i},
\end{equation}
where $I\in\N$, $a_i\in (0,\infty), p_i\neq q_i\in K$ and $\sum_{i=1}^{I}a_i=\norm{\mu_j^n}$. For every pair of points $p_i, q_i$ set $d_i:=d_{cc}(p_i,q_i)$, choose an arbitrary geodesic curve $\gamma_i\in\Gamma_{p_iq_i}$ and define the length measure $\nu_i\in\meas{\Gamma_K}$ as the pushforward of the Lebesgue measure $\lambda$ on $[0,d_i]$ through the curve $\gamma_i$, i.e. $\nu_i=(\gamma_i)_\sharp\lambda$ (recall that $\gamma_i$ is assumed to be arc-length parametrized). Finally, we consider the Hilbert space 
$$H_{(n,j)}:=\left(\sum_{i=1}^{I}L_2(\Gamma_K,\nu_i;\RR^s)\right)_{\ell_2}.$$ The reason for introducing this space is that the action of $\mu_j^n$ on a function $f\in\Lip_0(K)$ can be estimated by the norm of an upper gradient of $f$ in $H_{(n,j)}$, which we will use below. The idea is to transform $(f_j^n)$ and $(\mu_j^n)$ into an ``almost biorthogonal'' system by taking appropriate convex combinations of functions $f_j^n$ that we obtain from an application of Lemma \ref{lem:Hilbert_conv_comb} in Hilbert spaces $H_{(n,j)}$.

\medskip
\textbf{Iterative construction.}
Take a~sequence $(\varepsilon_k)_{k=1}^\infty\subset(0,1)$ such that 
\begin{equation}
\label{eq:sum epsilon}
\sum_{k=1}^\infty\varepsilon_k<\frac{\xi}{6C} .
\end{equation}
We will proceed inductively to construct an increasing sequence of natural numbers $(n_k)$ and sequences $(\sigma_k)\subset\mathfrak{M}\subset\free{K}$, 
$(g_k)\subset\mathcal{A}$ and $(h_k)\subset C^\infty\pare{\Omega,\enorm{\cdot}{d}}$ that will eventually yield witnesses of $\mathcal{K}$ not being a (V$^*$)-set. In each step, we will use Observation \ref{ob:index_set} to pass to a subset of the index set $\{(n,j):n>n_k,1\leq j\leq n\}$, that will be denoted the same again. The constructed objects will satisfy the following properties:
\begin{enumerate}[label={\upshape{(\roman*)}},itemsep=0.1cm]
    \item \label{biorthogonal} $\duality{g_k,\sigma_k}\geq\frac{\xi}{3}$;
    \item \label{small_evaluations_sigma}
    if $1\leq l< k$ and if $\sigma_k=\sum_{i=1}^{I}a_im_{p_iq_i}$ and $(\gamma_i)_{i=1}^I$ are the representation and the associated curves, respectively, fixed for $\sigma_k\in\mathfrak{M}$ in \eqref{eq:representation of mu}, then \begin{equation*}
\sum_{i=1}^I\frac{a_i}{d_i}\int_0^{d_i}\enorm{X(g_l)(\gamma_i(t))}{s}dt<C\varepsilon_l;
\end{equation*}
in particular, $\abs{\duality{g_l,\sigma_k}}<C\varepsilon_l$ (see below);
    \item \label{small_evaluations_mu} if $n>n_k$, $1\leq j\leq n$ and if $\mu_j^n=\sum_{i=1}^{I}a_im_{p_iq_i}$ and $(\gamma_i)_{i=1}^I$ are the representation and the associated curves, respectively, fixed for $\mu_j^n\in\mathfrak{M}$ in \eqref{eq:representation of mu}, then 
    \begin{equation*}
\sum_{i=1}^I\frac{a_i}{d_i}\int_0^{d_i}\enorm{X(g_k)(\gamma_i(t))}{s}dt<C\varepsilon_k;
\end{equation*}
in particular, $\abs{\duality{g_k,\mu_j^n}}<C\varepsilon_k$ (see below);
    \item \label{richness}
    $h_k\restrict_K\in\Lip_0(K)$ and
    $$\sup_{x\in \Gamma_K}\enorm{\prod_{j=1}^{k-1}\left(1-\enorm{X(g_j)(x)}{s}\right)X(g_k)(x)-X(h_k)(x)}{s}< \varepsilon_k$$ (with the convention that the empty product equals $1$).
\end{enumerate}
In \ref{small_evaluations_sigma} (and similarly in \ref{small_evaluations_mu}), the inequality $\abs{\duality{g_l,\sigma_k}}<C\varepsilon_l$ follows from Proposition \ref{prop:upper_gradient}, which ensures that $\enorm{X(g_l)}{s}$ is an upper gradient of $g_l$ on $\Omega$ and therefore
\begin{align}
\label{eq:g1_evaluated_by_mu_jn}
\left|\duality{g_l,\sigma_k}\right|&=\left|\sum_{i=1}^Ia_i\frac{g_l(p_i)-g_l(q_i)}{d_i}\right|\leq \sum_{i=1}^I\frac{a_i}{d_i}\left|g_l(p_i)-g_l(q_i)\right|\\\nonumber
&\leq \sum_{i=1}^I\frac{a_i}{d_i}\int_0^{d_i}\enorm{X(g_l)(\gamma_i(t))}{s}dt<C\varepsilon_l .
\end{align}

\medskip
\textbf{Initial step.}
Let us start by choosing the initial objects for $k=1$. Let $n_1\in\N$ be so large that $\frac {4}{\log n_1}<\varepsilon_1^2$. Now, for every $1\leq l\leq n_1$, $n> n_1$ and $1\leq j\leq n$, define functions $u_{l,{(n,j)}}:\Gamma_K\to (\RR^s)^{I}$ by
\begin{equation}
    \label{eq:definition_u_lnj}
u_{l,{(n,j)}}:=\left(\sqrt{\frac{a_1}{d_1}}\cdot X\left(f_l^{n_1}\right),\sqrt{\frac{a_2}{d_2}}\cdot X\left(f_l^{n_1}\right),\ldots,\sqrt{\frac{a_I}{d_I}}\cdot X(f_l^{n_1})\right)
\end{equation}
where $I, a_i, d_i$ are from the fixed representation \eqref{eq:representation of mu} of $\mu_j^n$.
Note that $$\norm{X(f_l^{n_1})}_{L_2(\Gamma_K,\nu_i;\RR^s)}^2\leq\norm{\nu_i}=d_i$$ by \ref{norm_of_X(f)}, so applying \eqref{eq:representation of mu} and \ref{norm_of_mu} we obtain that $u_{l,{(n,j)}}\in H_{(n,j)}$ with $$\norm{u_{l,{(n,j)}}}_{H_{(n,j)}}^2\leq \sum_{i=1}^Ia_i= \norm{\mu_j^n}\leq C.$$
So, for every $n>n_1$ and $1\leq j\leq n$, we may apply Lemma \ref{lem:Hilbert_conv_comb} to the collection of vectors $u_{1,(n,j)},u_{2,(n,j)},\ldots,u_{n_1,(n,j)}\in \sqrt{C}B_{H_{(n,j)}}$, and obtain sets $A,B\subset\{1,\ldots,n_1\}$ satisfying $\max A< \min B$, and such that
\begin{equation}\label{application_of_Lemma_1}
\norm{\frac{1}{|A|}\sum_{l\in A}u_{l,{(n,j)}}-\frac{1}{|B|}\sum_{l\in B}u_{l,{(n,j)}}}_{H_{(n,j)}}\leq\frac{4\sqrt{C}}{\sqrt{\log n_1}} < 2\sqrt{C}\varepsilon_1 .
\end{equation}
Note that $A$ and $B$ depend on $(n,j)$, but there are only finitely many choices for them. So, if we partition the index set $\set{(n,j):n>n_1,1\leq j\leq n}$ according to the choice of $(A,B)$ then, by Fact \ref{ob:index_set}, after a possible selection and relabeling of the indices $(n,j)$, we may fix common sets $A,B$ and assume that \eqref{application_of_Lemma_1} holds for every $n>n_1$ and $1\leq j\leq n$.

Now put
$$g_1:=\frac{1}{2|A|}\sum_{l\in A}f_l^{n_1}-\frac{1}{2|B|}\sum_{l\in B}f_l^{n_1},$$
$\sigma_1:=\mu_{\max A}^{n_1}\in\mathfrak{M}$, and $h_1:=g_1$. As a convex combination of functions in $\mathcal{A}$, $g_1$ also belongs to $\mathcal{A}$. Let us see that these objects satisfy properties \ref{biorthogonal}--\ref{richness}. The inequality \ref{biorthogonal} is implied by \ref{evaluations}, \ref{small_evaluations_sigma} holds true by vacuity, and \ref{richness} is trivial, so we only need to check \ref{small_evaluations_mu}.

Let now $n>n_1$ and $1\leq j\leq n$. For $I, a_i, p_i, q_i, \gamma_i$, and $\nu_i$ as in the fixed representation \eqref{eq:representation of mu} of $\mu_j^n$, we have 
\begin{align*}
\sum_{i=1}^I\frac{a_i}{d_i}&\int_0^{d_i}\enorm{X(g_1)(\gamma_i(t))}{s}dt\\
&=\sum_{i=1}^I\frac{a_i}{d_i}\int_{\Gamma_K}\enorm{X(g_1)(x)}{s}d\nu_i(x)\\
&\leq \left(\sum_{i=1}^Ia_i\right)^{\frac12}\cdot \left(\sum_{i=1}^I\frac{a_i}{d_i^2}\left(\int_{\Gamma_K}\enorm{X(g_1)(x)}{s}d\nu_i(x)\right)^2\right)^{\frac12}\\
&\leq \sqrt{C}\cdot \left(\sum_{i=1}^I\frac{a_i}{d_i^2}\left(\int_{\Gamma_K}\enorm{X(g_1)(x)}{s}d\nu_i(x)\right)^2\right)^{\frac12},
\end{align*}
where the penultimate line follows from Cauchy-Schwarz inequality and the last one from the optimality of the representation \eqref{eq:representation of mu} together with \ref{norm_of_mu}.
Denote by $U=(U_1,\ldots,U_I)\in H_{(n,j)}$ the function inside the norm in \eqref{application_of_Lemma_1}, i.e. $U_i:\Gamma_K\to\RR^s$ for every $i\in\{1,\ldots,I\}$.
Because the mapping $f\mapsto X(f)$ is linear, H\"older's inequality and \eqref{application_of_Lemma_1} allow us to derive the following estimate:
\begin{align*}
\sum_{i=1}^I\frac{a_i}{d_i^2}&\left(\int_{\Gamma_K}\enorm{X(g_1)(x)}{s}d\nu_i(x)\right)^2\\\nonumber
&\leq\sum_{i=1}^I\frac{a_i}{d_i^2}\,d_i\int_{\Gamma_K}\enorm{X(g_1)(x)}{s}^2d\nu_i(x)\\\nonumber
&=\sum_{i=1}^I\int_{\Gamma_K}\enorm{\sqrt{\frac{a_i}{d_i}}X(g_1)(x)}{s}^2d\nu_i(x)=\sum_{i=1}^I\norm{\frac 12 U_i}_{L_2(\Gamma_K,\nu_i;\RR^s)}^2\\\nonumber
&=\frac 14\norm{U}_{H_{(n,j)}}^2<C\varepsilon_1^2.
\end{align*}
Combining the two estimates above gives \ref{small_evaluations_mu}.

\medskip
\textbf{Inductive step.}
Now let $k>1$ and suppose that $n_i$, $g_i$, $\sigma_i$ and $h_i$ have already been found for $1\leq i\leq k-1$. Because the upper gradients $\enorm{X(g_i)}{s}\colon\Omega\to\RR$ are continuous on $\Omega$, so is the function $\prod_{i=1}^{k-1}\left(\mathbbm{1}_{\Omega}-\enorm{X(g_i)}{s}\right)\colon\Omega\to\RR$. Thus, by Stone-Weierstrass theorem, there exists a Lipschitz function $\psi_k:\Gamma_K\to \RR$ such that 
$$\sup_{x\in \Gamma_K}\left|\prod_{i=1}^{k-1}\left(1-\enorm{X(g_i)(x)}{s}\right)-\psi_k(x)\right|<\frac{\varepsilon_k}{6}.$$ 
Denote $L_k:=\Lip(\psi_k)<\infty$ and extend $\psi_k$ to $G$ with the same Lipschitz constant. Next, we apply Lemma \ref{lm:approximation} to $\psi_k$ to obtain a function $\varphi_k\in C^{\infty}\left(\Omega,\enorm{\cdot}{d}\right)$ with $\Lip(\varphi_k\restrict_K)\leq L_k$, $\sup_{x\in \Gamma_K}\enorm{X(\varphi_k)(x)}{s}\leq L_k$, and such that
\begin{equation}
\label{eq:definition_phi_k}
\sup_{x\in \Gamma_K}\left|\prod_{i=1}^{k-1}\left(1-\enorm{X(g_i)(x)}{s}\right)-\varphi_k(x)\right|<\frac{\varepsilon_k}{3}.
\end{equation}

Fix a finite set $\mathcal{N}_k\subset \Gamma_K$ which is $\eta$-dense in $\Gamma_K$ with respect to $\enorm{\cdot}{d}$, where $\eta>0$ satisfies 
 \begin{equation}
 \label{eq:density}
 L_k\beta\eta^{\alpha}<\frac{\varepsilon_k}{3}.
 \end{equation}
Now choose $n_k>n_{k-1}$ large enough depending on $|\mathcal{N}_k|, L_k$ and $\varepsilon_k$. More precisely, let $n_k$ satisfy
$$\frac{\left(6+2/\sqrt{C}\right)\sqrt{C+|\mathcal{N}_k|(L_k\beta R^{\alpha})^2}}{\sqrt{\log(n_k)}}<\varepsilon_k,$$
where the expression on the left-hand side is ``tailored to work well'' in computations leading to \ref{small_evaluations_mu} and in \eqref{computation_richness} below.

For each $1\leq l\leq n_k$, $n> n_k$ and $1\leq j\leq n$, we define the function $u_{l,{(n,j)}}:\Gamma_K\to (\RR^s)^{I}$ as in \eqref{eq:definition_u_lnj} but with $n_k$ in place of $n_1$, where $I, a_i, d_i$ are from the fixed representation \eqref{eq:representation of mu} of $\mu_j^n$. We also set the vector 
$$v_l:=\left(L_k\cdot f_l^{n_k}(x)\right)_{x\in\mathcal{N}_k}\,\in\RR^{|\mathcal{N}_k|}.$$
Then, appealing to \ref{norm_of_X(f)}, \eqref{eq:representation of mu}, \ref{norm_of_mu} and \ref{Holderness_of_f}, respectively, we obtain that
$$w_{l,(n,j)}:=(u_{l,(n,j)},\,v_l)\in H_{(n,j)}\oplus_2 \ell_2^{|\mathcal{N}_k|}$$
with the norm satisfying
$$\norm{w_{l,{(n,j)}}}_{H_{(n,j)}\oplus_2 \ell_2^{|\mathcal{N}_k|}}\leq \sqrt{C+|\mathcal{N}_k|(L_k\beta R^\alpha)^2}.$$
Hence, we can invoke Lemma \ref{lem:Hilbert_conv_comb} on elements $w_{1,{(n,j)}},\ldots,w_{{n_k},{(n,j)}}$ of the Hilbert space $H_{(n,j)}\oplus_2 \ell_2^{|\mathcal{N}_k|}$.
Here, similarly to the initial step, applying  Lemma \ref{lem:Hilbert_conv_comb} to the functions $(u_{l,(n,j)})_l$ will serve to find a subsequence in $\mathfrak{M}$ and a corresponding convex combination of functions $(f_l^{n_k})_l$ so that the evaluations are small and hence \ref{small_evaluations_mu} holds. On the other hand, using the lemma for vectors $(v_l)_l$ will yield that the convex combination of functions $(f_l^{n_k})_l$ is uniformly small on $\Gamma_K$. The latter will give us \ref{richness}, which will be crucial for constructing a weakly unconditionally Cauchy series.

As in the initial step, we apply Lemma \ref{lem:Hilbert_conv_comb} and Fact \ref{ob:index_set} to pass to a subsequence of indices $(n,j)$, where $n>n_k$ and $1\leq j\leq n$, and obtain sets $A,B\subset\{1,\ldots,n_k\}$ satisfying $\max A<\min B$ and such that for every $n>n_k$ and $1\leq j\leq n$, we have
\begin{align}
\label{Uk}
\norm{\frac{1}{|A|}\sum_{l\in A}w_{l,(n,j)}-\frac{1}{|B|}\sum_{l\in B}w_{l,(n,j)}}_{H_{(n,j)}\oplus_2 \ell_2^{|\mathcal{N}_k|}} &\leq \frac{4\sqrt{C+|\mathcal{N}_k|(L_k\beta R^\alpha)^2}}{\sqrt{\log (n_k)}} \nonumber \\
&< \frac{2\varepsilon_k}{3+1/\sqrt{C}} .
\end{align}
We set $$g_k:=\frac{1}{2|A|}\sum_{l\in A}f_l^{n_k}-\frac{1}{2|B|}\sum_{l\in B}f_l^{n_k}\qquad \textup{and}\qquad \sigma_k:=\mu_{\max A}^{n_k},$$ so that $g_k\in\mathcal{A}$ and $\sigma_k\in\mathfrak{M}\subset\free{K}$. 
It is easily checked that properties \ref{biorthogonal}--\ref{small_evaluations_mu} are satisfied. Indeed, \ref{biorthogonal} is implied by \ref{evaluations}, while \ref{small_evaluations_sigma} follows from assumption \ref{small_evaluations_mu} in the inductive hypothesis as $\sigma_k=\mu_j^n\in\mathfrak{M}$ for some $n>n_{k-1}$ and $1\leq j\leq n$. Property \ref{small_evaluations_mu} can be verified the same way as for $g_1$ in the initial step.

Last, set $h_k:=\varphi_kg_k$ where $\varphi_k$ was defined in \eqref{eq:definition_phi_k}. Note that $h_k\in  C^{\infty}(\Omega,\enorm{\cdot}{d})\cap\Lip_0(K)$ as both functions are smooth on $\Omega$, Lipschitz (and hence bounded) when restricted to $K$, and $g_k(x_0)=0$. Let us prove \ref{richness}. Since
$$X(h_k)=X(\varphi_k)g_k+\varphi_kX(g_k),$$
for every $x\in \Gamma_K$, we have 
\begin{align}
\label{computation_richness}
\nonumber
\Bigg\|\prod_{i=1}^{k-1}\left(1-\enorm{X(g_i)(x)}{s}\right)&X(g_k)(x)-X(h_k)(x)\Bigg\|_{\ell_2^{s}}\\\nonumber
&\leq\frac{\varepsilon_k}{3}\enorm{X(g_k)(x)}{s}+\enorm{X(\varphi_k)(x)g_k(x)}{s}\\\nonumber
&\leq\frac{\varepsilon_k}{3}\enorm{X(g_k)(x)}{s}+L_k|g_k(x)|\\
&\leq\frac{\varepsilon_k}{3}+L_k|g_k(x)| .
\end{align}
Here, the first two inequalities follow from the properties of function $\varphi_k$, while the third one comes from \ref{norm_of_X(f)}. Now choose $y\in\mathcal{N}_k$ such that $\enorm{x-y}{d}\leq\eta$, then
\begin{align*}
L_k\abs{g_k(x)} &\leq L_k\abs{g_k(y)} + L_k\abs{g_k(x)-g_k(y)} \\
&\leq \frac{\varepsilon_k}{3+1/\sqrt{C}} + L_k\beta\enorm{x-y}{d}^\alpha \\
&\leq \frac{\varepsilon_k}{3} + L_k\beta\eta^\alpha < \frac{2\varepsilon_k}{3}
\end{align*}
where the second inequality is inferred from \ref{Holderness_of_f} and the part of estimate \eqref{Uk} involving vectors $v_l\in\ell_2^{|\mathcal{N}_k|}$. Combining these estimates yields \ref{richness}, and thus concludes the inductive construction.

\medskip
\textbf{Construction of the WUC series.}
Finally, we claim that the series $\sum_{n=1}^\infty h_n\restrict_K$ obtained by \ref{richness} is a WUC series in $\Lip_0(K)$ witnessing that $\mathcal{K}$ is not a (V$^*$)-set. For better clarity, we will relax the notation and write $h_n$ instead of $h_n\restrict_K$ also for the functions in $\Lip_0(K)$; the domain should always be clear from the context.
Observe that by \ref{norm_of_X(f)}, for any $x\in \Gamma_K$ we have
\begin{align*}
    \sum_{n=1}^{\infty}&\enorm{\prod_{j=1}^{n-1}\left(1-\enorm{X(g_j)(x)}{s}\right)X(g_n)(x)}{s}=\sum_{n=1}^{\infty}\prod\limits_{j=1}^{n-1}\left(1-\enorm{X(g_j)(x)}{s}\right)\enorm{X(g_n)(x)}{s}\\
&=\enorm{X(g_1)(x)}{s}+\sum_{n=2}^{\infty}\Biggl(\prod\limits_{j=1}^{n-1}\left(1-\enorm{X(g_j)(x)}{s}\right)-\prod\limits_{j=1}^{n}\left(1-\enorm{X(g_j)(x)}{s}\right)\Biggr)\leq 1,
\end{align*}
where we use the convention that the empty product is 1. Thus, by \ref{richness} and \eqref{eq:sum epsilon}, 
\begin{equation}
    \label{sum_of_X(h)}
\sum_{n=1}^\infty \enorm{X(h_n)(x)}{s}\leq 1+\sum_{n=1}^\infty \varepsilon_n<1+\frac{\xi}{6C}.
\end{equation}
Because $h_n\in C^\infty(\Omega,\enorm{\cdot}{d})$, the function $\enorm{X(h_n)}{s}\colon\Omega\to\RR$ is an upper gradient of $h_n$ by Proposition \ref{prop:upper_gradient}. So, we can estimate the action of any elementary molecule supported in $K$ on the function $h_n$ by the integral of $\enorm{X(h_n)}{s}$ along any geodesic curve joining the points of the molecule as follows. Let $p\neq q\in K$ and $\gamma\in\Gamma_{pq}$, then by \eqref{sum_of_X(h)}
\begin{align*}
\sum_{n=1}^\infty \abs{\duality{h_n,m_{pq}}}&\leq \sum_{n=1}^\infty \frac{1}{d_{cc}(p,q)}\int_0^{d_{cc}(p,q)}\enorm{X(h_n)(\gamma(t))}{s}dt\\
&= \frac{1}{d_{cc}(p,q)}\int_0^{d_{cc}(p,q)}\sum_{n=1}^\infty\enorm{X(h_n)(\gamma(t))}{s}dt\\
&< 1+\frac{\xi}{6C}.
\end{align*}
Note that, importantly, the bound does not depend on $p,q\in K$. In particular, 
\begin{equation}
    \label{eq:norm_of_h_n}
    \Lip(h_n)\leq 1+\frac{\xi}{6C} .
\end{equation}
Since for every element $m\in\free{K}$ there exists a representation $m=\sum_{i=1}^{\infty}a_im_{p_iq_i}$ with $a_i\geq 0$, $p_i\neq q_i\in M$ such that $\sum_{i=1}^\infty a_i\leq2\norm{m}$, as in \eqref{eq:sum of molecules}, by simply changing the order of summation we infer that
$$\sum_{n=1}^\infty\abs{\duality{h_n,m}}\leq 2\left(1+\frac{\xi}{6C}\right)\norm{m} .$$
Thus $\sum_{n=1}^\infty h_n$ is weakly unconditionally Cauchy against elements of the predual $\lipfree{K}$, and thus WUC by Fact \ref{fact:wuc}\ref{w*uc}.

It remains to be shown that $\inf_n\sup_{\kappa\in\mathcal{K}}\abs{\duality{h_n,\kappa}}>0$. Let $n\in\N$. For any $x\in \Gamma_K$, we have
\begin{align}
\label{eq:X(h-g)}\nonumber
\enorm{X(h_n-g_n)(x)}{s}&\leq \enorm{X(h_n)(x)-\prod_{j=1}^{n-1}\left(1-\enorm{X(g_j)(x)}{s}\right)X(g_n)(x)}{s}\\\nonumber
&\phantom{\leq}+ \enorm{\prod_{j=1}^{n-1}\left(1-\enorm{X(g_j)(x)}{s}\right)X(g_n)(x)-X(g_n)(x)}{s}\\\nonumber
&\leq \varepsilon_n+\left(1-\prod_{j=1}^{n-1}\left(1-\enorm{X(g_j)(x)}{s}\right)\right)\enorm{X(g_n)(x)}{s}\\
&\leq \varepsilon_n+\sum_{j=1}^{n-1}\enorm{X(g_j)(x)}{s},
\end{align}
where the second inequality follows from \ref{richness} and the third one from \ref{norm_of_X(f)} combined with the general inequality $1-\prod_{j=1}^n(1-c_j)\leq \sum_{j=1}^n c_j$ that holds true for any $(c_j)_{j=1}^n\subset [0,1]$. Now, because both $h_n,g_n\in C^\infty(\Omega,\enorm{\cdot}{d})$, the function $\enorm{X(h_n-g_n)}{s}$ is an upper gradient of $h_n-g_n$ by Proposition \ref{prop:upper_gradient} again, and we have 
\begin{align*}
\abs{\duality{h_n-g_n,\sigma_n}}&\leq \sum_{i=1}^I\frac{a_i}{d_i}\int_0^{d_i}\enorm{X(h_n-g_n)(\gamma_i(t))}{s}dt\\
&\leq \sum_{i=1}^I\frac{a_i}{d_i}\int_0^{d_i}\left(\varepsilon_n+\sum_{j=1}^{n-1}\enorm{X(g_j)(\gamma_i(t))}{s}\right)dt\\
&=\varepsilon_n\sum_{i=1}^I a_i +\sum_{j=1}^{n-1}\sum_{i=1}^I\frac{a_i}{d_i}\int_0^{d_i}\enorm{X(g_j)(\gamma_i(t))}{s}dt\\
&\leq C\varepsilon_n+ C\sum_{j=1}^{n-1}\varepsilon_j=C\sum_{j=1}^{n}\varepsilon_j<\frac{\xi}{6},
\end{align*}
according to the above estimate \eqref{eq:X(h-g)}, the bound of the norm of $\sigma_n$, \ref{small_evaluations_sigma}, and \eqref{eq:sum epsilon}; here, $I, a_i, d_i, \gamma_i$ are from the fixed representation \eqref{eq:representation of mu} of $\sigma_n$.
So, applying \ref{biorthogonal}, we obtain
\begin{align*}
\duality{h_n,\sigma_n}&\geq \duality{g_n,\sigma_n}-\abs{\duality{h_n-g_n,\sigma_n}}>\frac{\xi}{6}.
\end{align*}
To conclude, by \ref{approximation} and \eqref{eq:norm_of_h_n}, we may find a $\kappa\in\mathcal{K}$ so that
$$\duality{h_n,\kappa} \geq \duality{h_n,\sigma_n} - \Lip(h_n)\norm{\sigma_n-\kappa} > \frac{\xi}{12} .$$ The proof of the theorem is thus complete.
\end{proof}

\section*{Acknowledgements}

The authors would like to thank Michal Doucha for answering their questions about Carnot groups, Chris Gartland for insightful discussions on Carnot-Carath\'eodory metric spaces and for informing them of the failure of condition \ref{wuc extension} in Proposition \ref{pr:sufficient conditions} for $\ell_1$, and Anton\'in Proch\'azka for providing Corollary \ref{cr:compact determination subspace}.

Some of this research was carried out during visits of the first author to the Faculty of Information Technology at the Czech Technical University in Prague in 2022 and 2024.

%\section*{Declarations}

%\subsection*{Funding}
R. J. Aliaga was partially supported by Grant PID2021-122126NB-C33 funded by MICIU/AEI/10.13039/501100011033 and by ERDF/EU. E. Perneck\'a was supported by the grant GA\v CR 22-32829S of the Czech Science Foundation. A. Quero was supported by Grant PID2021-122126NB-C31 funded by MICIU/AEI/10.13039\allowbreak/501100011033 and ERDF/EU, and the CTU Global Postdoc Fellowship program.

%\subsection*{Competing Interests}
%The authors have no relevant financial or non-financial interests to disclose.

%\subsection*{Data availability}
%Data sharing is not applicable to this article as no datasets were generated or analysed during the current study.

\end{document}